\documentclass[reqno]{amsart}
\usepackage{amssymb}
\usepackage{hyperref}
\usepackage{graphicx,color}
\usepackage[section]{algorithm}
\usepackage{algorithmic}
\usepackage{enumerate}
\usepackage{mathrsfs,enumerate}
\usepackage{multirow,array}
\usepackage{subfigure}
\usepackage{pstricks,pst-node,pst-tree, pst-plot, pst-eps}
\usepackage{xspace}
\usepackage{marginnote}
\usepackage{upgreek}
\usepackage{bm}
\usepackage{comment}
\usepackage{tikz}
\usetikzlibrary{matrix}

\chardef\atsign='100

\newcommand{\Forall}{\quad\text{for all }}

\newcommand{\dH}{\dot H}

\newcommand{\cQ}{\mathcal Q}
\newcommand{\bv}{\mathbf v}
\newcommand{\bx}{\mathbf x}

\newcommand{\bP}{\mathbf P}
\newcommand{\I}{\mathrm I}
\newcommand{\II}{\mathrm{II}}
\newcommand{\bsxi}{\boldsymbol{\xi}}

\newcommand{\bsz}{\boldsymbol{z}}

\renewcommand{\Re}{\mathfrak{Re}}
\renewcommand{\Im}{\mathfrak{Im}}

\DeclareMathOperator*{\argmin}{\arg\!\min}

\theoremstyle{plain}
\newtheorem{theorem}{Theorem}[section]
\newtheorem{lemma}[theorem]{Lemma} 
\newtheorem{proposition}[theorem]{Proposition}

\newtheorem{remark}{Remark}[section]

\graphicspath{{figures/}}

\newcommand{\DG}[1]{{\color{cyan}{#1}}}

\newcommand{\WLc}[1]{{\color{red}\textbf{\textsf{[WL: #1]}}}}
\newcommand{\DGc}[1]{{\color{cyan}\textbf{\textsf{[DG: #1]}}}}

\author[A.~Bonito]{Andrea Bonito}
\address[A.~Bonito]{Department of Mathematics, Texas A\&M University, College Station, TX 77843, USA}
\email{bonito@math.tamu.edu}

\author[D.~Guignard]{Diane Guignard}
\address[D.~Guignard]{Department of Mathematics and Statistics, University of Ottawa, Ottawa, ON K1N 6N5, Canada}
\email{dguignar@uottawa.ca}

\author[W.~Lei]{Wenyu Lei}
\address[W.~Lei]{School of Mathematical Sciences, University of Electronic Science and Technology of China, No. 2006 Xiyuan Ave., West Hi-Tech Zone, 611731, Chengdu, China}
\email{wenyu.lei@uestc.edu.cn}

\thanks{A.B. is partially supported by the NSF Grant DMS-2110811. D.G. acknowledges the support provided by the Natural Science and Engineering Research Council (NSERC, grant RGPIN-2021-04311). W.L. is partially supported by the National Natural Science Foundation of China (Grant No. 12301496).}

\begin{document}

\title[Numerical Approx of GRFs]
{Numerical Approximation of Gaussian random fields on Closed Surfaces}

\date{\today}


\maketitle

\begin{abstract}
     We consider the numerical approximation of Gaussian random fields on closed surfaces defined as the solution to a fractional stochastic partial differential equation (SPDE) with additive white noise. The SPDE involves two parameters controlling the smoothness and the correlation length of the Gaussian random field. The proposed numerical method relies on the Balakrishnan integral representation of the solution and does not require the approximation of eigenpairs. Rather, it consists of a sinc quadrature coupled with a standard surface finite element method. We provide a complete error analysis of the method and illustrate its performances in several numerical experiments.

\end{abstract}

\section{Introduction}

Random fields appear in many applications. A typical example is when uncertainty is incorporated in mathematical models. The uncertainty in the parameters reflects an intrinsic variability of the system or our inability to adequately characterize all the data, for instance due to measurements. In a probability setting, the uncertainty is modeled by random variables or more generally random fields. A concrete example in the context of partial differential equations (PDEs) is Darcy's flow, where the permeability is modeled as a log-normal random field \cite{BN2014,NT2015}.
Many methods have been developed to solve stochastic partial differential equations (SPDEs) or to compute statistics of their solutions. The most commonly used are Monte Carlo-type methods \cite{GKNSS2011,CGST2011}, the stochastic Galerkin method \cite{GS1991,BTZ2004} and the stochastic collocation method \cite{XH2005,BNT2007}. Common to all these methods is the use of samples from random fields.

The standard approach to evaluate a random field is to use a truncated series expansion, e.g. a Karhunen--Lo\`eve (KL) \cite{L1977,L1978} or Fourier expansion, thereby consisting of only a finite number of random variables (this process is usually referred to as the \emph{finite dimensional noise assumption} in the random PDEs literature).

The truncated KL expansion is widely used in practice as it is the one minimizing the mean square error. In specific cases, the analytic expression of the KL expansion is known, e.g. for random fields with separable exponential covariance function on hyperrectangles (the eigenfunctions being (product of) sine and cosine functions) \cite{KL2010,LPS2014} or for isotropic Gaussian random fields (GRFs) on the sphere. Here the eigenfunctions are the spherical harmonics \cite{MP2011,LS2015} and are known analytically. However, in general, the computation of KL expansions requires approximating eigenfunctions of an elliptic operator.
We refer to \cite{ST2006} for an efficient algorithm (based on discontinuous finite elements) that can be used to approximate the truncated KL expansion of a random field with prescribed mean field and covariance function defined on a Euclidean domain. In the case of GRFs on compact Riemannian manifolds, a numerical method combining a Galerkin approximation of the Laplace--Beltrami operator with a  truncated Chebyshev series approximation of the square-root of a covariance matrix is introduced in \cite{LP2021}.


As an alternative to KL expansions, multilevel representations have been considered for efficient simulation of random fields. We refer for instance to \cite{BD2022} for (modified) spherical needlet-based representations of GRFs on the sphere or \cite{HHKS2021} for wavelet-based representations of GRFs indexed by Euclidean domains, manifolds, or graphs.

In this work, we are interested in GRFs which belong to the so-called (Whittle--) Mat\'ern class, see \eqref{def:Matern}. They are characterized as the solution to a fractional PDE with additive white noise of the form
\begin{equation} \label{def:SPDE}
     (\kappa^2I-\Delta_{\gamma})^s\widetilde u = \widetilde w,
\end{equation}
where $\kappa >0$, $s>\frac{n-1}{4}$ and $\gamma$ is a closed hypersurface in $\mathbb{R}^n$ ($n=2,3$).

This link between GRFs and fractional SPDEs was already observed in \cite{W1954,W1963} for stationary fields on $\mathbb{R}^n$ and later extended in \cite{lindgren2011explicit} to more general fields, including GRFs on bounded domains. Numerical solvers for such fractional SPDEs can then be used to approximate random fields, see for instance \cite{BKK2018,BK2020, cox2020regularity} for GRFs defined in Euclidean domains, \cite{jansson2021surface,BTMD2020} for GRFs on general manifold, and \cite{BSW2022} for GRFs on graphs. A list of recent applications of the SPDE approach to the approximation of random fields can be found in \cite{LBR2022}.


We propose a numerical method for approximating the solution to \eqref{def:SPDE} when $\frac{n-1}{4}<s<1$; refer to Remark~\ref{r:recursion} for extensions to $s \geq 1$. The method relies on the integral representation
\begin{equation} \label{eqn:integral_repres}
     (\kappa^2 I-\Delta_{\gamma})^{-s} \widetilde w = \frac{\sin(\pi s)}{\pi}\int_0^\infty \mu^{-s}((\mu+\kappa^2) I -\Delta_\gamma)^{-1}\widetilde w \, d\mu.
\end{equation}
The improper integral is approximated by a sinc quadrature and a continuous linear finite element method on an approximate surface $\Gamma$ is put forward to approximate the integrand at each sinc quadrature point.

Such a two-step numerical strategy (sinc quadrature coupled with a finite element method) was originally proposed in \cite{bonito2015numerical} for Euclidean domains and $\kappa = 0$. We refer to \cite{BP15,BLP17,antil2018fractional} for several extensions and notably to \cite{bonito2021approximation} when the domain is a closed hypersurfaces.

All the above mentioned work on fractional PDEs are for deterministic data.
Although the approximation of the solution to \eqref{def:SPDE} considered in this work is similar to the one introduced in \cite[Section~3]{bonito2021approximation}, realizations of the stochastic right-hand side $\widetilde w$ are not almost surely in $L^2(\gamma)$.
Whence, following \cite{herrmann2020multilevel}, we replace $\widetilde w$ by a projection onto a conforming finite element space defined on $\gamma$. The resulting approximation is a Gaussian random vector with zero mean and a covariance matrix given by a weighted finite element mass matrix.
Alternatively in \cite{bolin2020numerical}, which considers the Euclidean setting, $\widetilde w$ is approximated by an expansion with respect to the discrete eigenfunctions of the conforming finite element approximation of the operator $(\kappa^2 I-\Delta_\gamma)$. Using a change of basis, the resulting approximation to $\widetilde w$ can be expressed in terms of the finite element basis functions and can then be used in practice since both are equal in the mean square sense. This property is critical in our analysis of the error in the mean square norm.

We now comment on the main novelties of this work.

\begin{itemize}
     \item Compared to \cite{bolin2020numerical,herrmann2020multilevel}, we improve on the convergence of the sinc quadrature by deriving an exponential convergence rate with respect to the quadrature spacing and in particular not depending on the finite element mesh size; see Theorem~5.2 and~6.1 as well as the numerical validation in Remark~\ref{r:sinc-independent}. This implies that the mesh size and the quadrature spacing can be chosen independently.
     \item The abstract analysis proposed in \cite{herrmann2020multilevel} considers a finite element method defined on the exact surface and provides strong error estimates. 
     Instead, we design and analyze a parametric finite element method on approximate surfaces by taking into account the geometric error. Note that setting the finite element on approximate (polygonal) surfaces facilitates the construction of the discrete system.
     In \cite{jansson2021surface}, the surface is restricted to a sphere but strong and mean squared norm error estimates are derived for parametric finite element methods on approximations of the sphere. In that case, the approximation of the white noise strongly relies on the knowledge of the eigenpairs for the sphere. Our analysis, accounts for the intricate influence of the geometric error in the approximation of the eigenpairs for general $C^3$ surfaces.
     \item Borrowing ideas from \cite{bolin2020numerical}, we estimate the discrepancy between the mean square norm of the exact solution and that of the proposed approximate solution. Our analysis relies on the asymptotic behavior of the eigenvalues of the Laplace--Beltrami operator (Weyl's law) as well as the ability for the finite element method to approximate the eigenvalues of $(\kappa^2 I - \Delta_\gamma)$. However, the analysis in \cite{bolin2020numerical} relies on an assumption for the approximation of individual eigenvalues (Assumption~2.6 in \cite{bolin2020numerical}) 
     \begin{equation}\label{o:ass}
     0< \widetilde \Lambda^\gamma_j - \widetilde \lambda_j \leq C \widetilde \lambda_j^r h^q, \qquad q>1, \qquad r>0,
     \end{equation}
     where $\widetilde \Lambda^\gamma_j$ is the approximation of the $j$th eigenvalue $\widetilde \lambda_j$ (see \eqref{e:eigenpb} and \eqref{e:dicrete_eig}). 
     To derive optimal convergence rates in the mean square norm, it appears that one needs $q=r=2$. With this choice of parameters, the constant $C$ in \eqref{o:ass} may not be uniform with in $j$; see e.g. \cite[Theorem~47.10]{ern2021finite}. 
     In this work, we do not rely on such assumption but rather derive a weaker but  cluster robust eigenvalue error estimate (Lemma~\ref{l:eigen-error}) instrumental for the optimal convergence rates in the mean square norm provided by Theorem~\ref{t:weak-convergence}. It is worth pointing out that the cluster robust estimate of Lemma~\ref{l:eigen-error} is used in \cite[Appendix~C]{LP2021} to guarantee \eqref{o:ass} with a uniform constant $C$ when $q=r=2$.
     
\end{itemize}

This paper is organized as follows. We start by recalling in Section~\ref{sec:RFs} known results about random fields, including properties of their KL expansions. In Section~\ref{s:model}, we define our model problem \eqref{e:spde} in Sobolev spaces and introduce the Dunford--Taylor integral representation for the solution. The numerical method is introduced in Section~\ref{s:simulation}. The strong and mean square norm error estimates are discussed in Sections~\ref{s:analysis} and \ref{s:weak}, respectively.
Theorem~\ref{t:strong} (strong convergence) and Theorem~\ref{t:weak-convergence} (convergence in the mean square norm) are the main results of this work. Several numerical experiments illustrating the efficiency of the method and the influence of $s$ and $\kappa$ on resulting GRFs are provided in Section~\ref{s:numeric}.



\section{Random Fields} \label{sec:RFs}

We start this section by recalling the notion of random fields. A natural approach to represent a random field is to consider its Karhunen--Lo\`eve expansion formally described in Section~\ref{ss:rm}, see \cite{KL2010} and references therein. However, this representation is computationally challenging as it requires (an approximation of) the eigenpairs of the associated covariance operator, see \eqref{def:cov} below.
Alternatively, Mat\'ern Gaussian random fields defined in Section~\ref{ss:grf} and considered in this work are solutions to an SPDE with Gaussian white noise. This representation is critical for the efficient numerical method analyzed in this work.


\subsection{Random Fields in Euclidean Domains}\label{ss:rm}
Let $D\subset\mathbb{R}^n$, $n=1,2,3$, be a bounded domain and let $(\Omega,\mathcal{F},\mathbb{P})$ be a complete probability space, where $\Omega$ is the set of outcomes, $\mathcal{F}\subset 2^{\Omega}$ is a $\sigma$-algebra of events, and $\mathbb{P}: \mathcal{F}\rightarrow [0,1]$ is a probability measure. Let $u:D\times\Omega\rightarrow\mathbb{R}$ be a real-valued random field, namely $u(\bx,\cdot)$ is a real-valued random variable for each $\bx\in D$.
We suppose that $u\in L^2(\Omega)\otimes L^2(D)$, that is $u(\cdot,\omega)\in L^2(D)$ $\mathbb{P}$-a.e. in $\Omega$ and $u(\bx,\cdot)\in L^2(\Omega)$ a.e. in $D$. With a slight abuse of notation\footnote{For any Hilbert space $H$ on $D$, the tensor product space $L^2(\Omega)\otimes H$ is isomorphic to $L^2(\Omega;H)$ defined in \eqref{def:Bochner}, see for instance \cite{BTZ2004}.}, we shall often consider $u$ as a square integrable mapping $u:\Omega\rightarrow L^2(D)$.
This means that $u\in L^2(\Omega;L^2(D))$, where for any Hilbert space $H$ on $D$ equipped with the norm $\|\cdot\|_H$, the Bochner space $L^2(\Omega;H)$ is defined by
\begin{equation} \label{def:Bochner}
L^2(\Omega;H):=\left\{v:\Omega\rightarrow H:\,\, v \text{ is strongly measurable and} \,\, \|v\|_{L^2(\Omega;H)}<\infty\right\}    
\end{equation}
with
$$\|v\|_{L^2(\Omega;H)}^2:=\mathbb{E}\left[\|v\|_H^2\right]:=\int_{\Omega}\|v(\omega)\|_H^2d\mathbb{P}(\omega)<\infty.$$

We introduce $\bar u:D\rightarrow\mathbb{R}$ and ${\rm cov}_u:D\times D\rightarrow\mathbb{R}$ the mean and covariance functions of $u$ defined respectively by
$$\bar u(\bx):=\mathbb{E}[u(\bx,\cdot)] \quad \mbox{for a.e.} \,\,\bx\in D,$$
and
\begin{equation} \label{def:cov}
     {\rm cov}_u(\bx,\bx'):=\mathbb{E}[(u(\bx,\cdot)-\bar u(\bx))(u(\bx',\cdot)-\bar u(\bx'))] \quad \mbox{for a.e.} \,\,\bx,\bx'\in D.
\end{equation}

\noindent We say that $u$ is \emph{weakly stationary} if the mean field $\bar u$ is constant and the covariance function depends only on the difference $\bx-\bx'$, namely ${\rm cov}_u(\bx,\bx')=c(\bx-\bx')$ for some function $c:D\rightarrow\mathbb{R}$. If in addition ${\rm cov}_u(\bx,\bx')$ depends only on the Euclidean distance $|\bx-\bx'|$ then $u$ is said to be \emph{isotropic}; see \cite{AT2007}.

To motivate the representation of a random field by its Karhunen--Lo\`eve expansion, see \cite{KL2010} and references therein, we suppose that $u$ is such that ${\rm cov}_u$ in \eqref{def:cov} is continuous on $\bar D\times \bar D$. This is for instance the case when $u$ is a Gaussian random field with Mat\'ern covariance, see Section~\ref{ss:grf} below, and $D$ is convex; see for instance \cite[Lemma 11]{NT2015}. Then $u$ can be represented by its Karhunen--Lo\`eve expansion
\begin{equation} \label{eqn:KL}
     u(\bx,\omega) = \bar u(\bx) + \sum_{i=1}^{\infty}\sqrt{\lambda_i}\xi_i(\omega)\phi_i(\bx),
\end{equation}
where the sum converges in $L^2(\Omega;L^2(D))$. In \eqref{eqn:KL}, $\{(\lambda_i,\phi_i)\}_{i=1}^{\infty}$ with $$\lambda_1\ge\lambda_2\ge\ldots\rightarrow 0 \quad \mbox{and} \quad \int_D\phi_i(\bx)\phi_j(\bx)d\bx=\delta_{ij}$$
are the eigenpairs of the (compact, self-adjoint, positive semi-definite) Hilbert-Schmidt integral operator $C_u:L^2(D)\rightarrow L^2(D)$ defined for $v\in L^2(D)$ by
$$(C_uv)(\bx):=\int_{D}{\rm cov}_u(\bx,\bx')v(\bx')d\bx', \quad \bx\in D.$$
That is, the eigenpair $(\lambda_i,\phi_i)$, $i\in\mathbb{N}$, satisfies the following eigenvalue problem (Fredholm integral equation of the second kind)
\begin{equation} \label{def:EVP_KL}
     \int_{D}{\rm cov}_u(\bx,\bx')\phi_i(\bx')d\bx' = \lambda_i\phi_i(\bx), \quad \bx\in D.
\end{equation}
Moreover, the random variables $\{\xi_i\}_{i=1}^\infty$, which are given by
\begin{equation} \label{def:RV_KL}
     \xi_i(\omega):=\frac{1}{\sqrt{\lambda_i}}\int_{D}(u(\bx,\omega)-\bar u(\bx))\phi_i(\bx)d\bx,
\end{equation}
have zero mean, unit variance, and are pairwise uncorrelated:
$$\mathbb{E}[\xi_i]=0 \quad \mbox{and} \quad \mathbb{E}[\xi_i\xi_j]=\delta_{ij}.$$
In other words they are orthonormal in $L^2(\Omega)$.

\subsubsection{Properties of the KL Expansion}
The KL series in \eqref{eqn:KL} converges not only in $L^2(\Omega;L^2(D))$ but also in $L^2(\Omega)$ pointwise 
in $\bx\in D$. Indeed, since ${\rm cov}_u$ is symmetric, continuous, and positive semi-definite, thanks to Mercer's theorem \cite{M1909} we have
\begin{equation} \label{eqn:expansion_cov_u}
     {\rm cov}_u(\bx,\bx')=\sum_{i=1}^{\infty}\lambda_i\phi_i(\bx)\phi_i(\bx'),
\end{equation}
where the convergence is absolute and uniform on $D\times D$, and thus
$${\rm Var}\left[u(\bx,\cdot)\right]={\rm cov}_u(\bx,\bx)=\sum_{i=1}^{\infty}\lambda_i\phi_i(\bx)^2.$$
Now for any $N\in\mathbb{N}$, let
\begin{equation} \label{eqn:truncated_KL}
     u^N(\bx,\omega):=\bar u(\bx)+\sum_{i=1}^N\sqrt{\lambda_i}\xi_i(\omega)\phi_i(\bx)
\end{equation}
denote the truncated KL expansion with $N$ terms, which satisfies the following relation for any $\bx\in D$
$$
     \begin{aligned}
          {\rm Var}\left[u^N(\bx,\cdot)\right] & =\mathbb{E}\left[(u^N(\bx,\cdot)-\bar u(\bx))^2\right]                                                                                                        \\
                                               & =\sum_{i,j=1}^{N}\sqrt{\lambda_i}\sqrt{\lambda_j}\phi_i(\bx)\phi_j(\bx)\underbrace{\mathbb{E}[\xi_i\xi_j]}_{=\delta_{ij}}=\sum_{i=1}^N\lambda_i\phi_i(\bx)^2.
     \end{aligned}
$$
Then for any $\bx\in D$, thanks to Fubini's theorem, relation \eqref{eqn:expansion_cov_u}, and the fact that $\{\phi_i\}_{i=1}^{\infty}$ forms an orthonormal basis of $L^2(D)$, we infer that
\begin{eqnarray}
     \mathbb{E}\left[(u(\bx,\cdot)-u^N(\bx,\cdot))^2\right] & = & \mathbb{E}\left[(u(\bx,\cdot)-\bar u(\bx)+\bar u(\bx)-u^N(\bx,\cdot))^2\right] \nonumber \\
     & = & {\rm Var}(u(\bx,\cdot))-{\rm Var}(u^N(\bx,\cdot)) \nonumber \\
     & = & \sum_{i=N+1}^{\infty}\lambda_i\phi_i(\bx)^2\rightarrow 0 \quad \mbox{as } N\rightarrow\infty \label{eqn:MS_truncation}
\end{eqnarray}
uniformly in $\bx\in D$, from which we deduce that \eqref{eqn:KL} holds for all $\bx\in D$.

Using the fact that $\|\phi_i\|_{L^2(D)}=1$ for all $i\in\mathbb{N}$ together with Fubini's theorem, we easily deduce from \eqref{eqn:MS_truncation} the relation
$$\|u-u^N\|_{L^2(\Omega;L^2(D))}^2=\sum_{i=N+1}^{\infty}\lambda_i.$$
Therefore, the mean square error between the random field $u$ and its truncated KL expansion $u^N$ is controlled by the decay of the eigenvalues, and the latter depends on the smoothness of the covariance function. Estimates on the decay of the eigenvalues can be found for instance in \cite{FST2005}.

The truncated KL expansion \eqref{eqn:truncated_KL} is the one minimizing the mean square error in the sense that (cf. \cite{GS1991})
$$
     \begin{aligned}
          \{(\sqrt{\lambda_i}\xi_i & ,\phi_i)\}_{i=1}^N                                                                                                                                                                      \\
                                   & = \argmin_{\{(\zeta_i,\psi_i)\}_{i=1}^N, \int_D\psi_i\psi_j=\delta_{ij}}\mathbb{E}\left[\int_D\left(u(\bx,\cdot)-\bar u(\bx)-\sum_{i=1}^N\psi_i(\bx)\zeta_i(\cdot)\right)^2d\bx\right].
     \end{aligned}
$$


Finally, note that the KL representation \eqref{eqn:KL} of $u$ only requires the knowledge of the mean field and the covariance function, i.e. of the first two moments of $u$. From a modeling point of view, random fields can be generated by prescribing functions $\bar u$ and ${\rm cov}_u$ and building the representation \eqref{eqn:KL} or \eqref{eqn:truncated_KL}. However, the difficulty lies in the fact that not any function ${\rm cov}_u$ is admissible, in particular it needs to be positive semi-definite. We will consider in the next section a specific family of admissible covariance functions, see \eqref{def:Matern} below, and we refer to \cite{DR2007} for other classes of functions.

\subsubsection{Gaussian Random Fields}\label{ss:grf}
It is well-known that if $u$ is a GRF\footnote{We say that $u$ is a Gaussian random field if the random vector $(u(\bx_1),u(\bx_2),\ldots,u(\bx_M))$ follows a multivariate Gaussian distribution for any $\bx_1,\bx_2,\ldots,\bx_M\in D$ and any $M\in\mathbb{N}$.} then $u$ is fully determined by its mean and covariance functions \cite{LPS2014}. Moreover, the KL expansion of a GRF is given by \eqref{eqn:KL}, where the $\xi_i$ are pairwise independent and follow a normal distribution, i.e. $\xi_i\stackrel{i.i.d.}{\sim}\mathcal{N}(0,1)$.

In particular, we say that a GRF belongs to the Mat\'ern family \cite{M1960} if the covariance function reads
\begin{equation} \label{def:Matern}
     {\rm cov_u}(\bx,\bx') = \frac{\sigma^2}{2^{\nu-1}\Gamma(\nu)}\left(\kappa|\bx-\bx'|\right)^{\nu}K_{\nu}\left(\kappa|\bx-\bx'|\right), 
\end{equation}
where $\sigma^2$ is the marginal variance, $K_{\nu}$ is the modified Bessel function of the second kind of order $\nu$, $\Gamma$ is the gamma function, and $\kappa$ and $\nu$ are positive parameters controlling the spatial correlation range and the smoothness of ${\rm cov}_u$, respectively. Specifically, the underlying random field $u$ is $\lceil\nu\rceil-1$ mean square times differentiable. The covariance function \eqref{def:Matern} is usually re-parameterized using, for instance, $\kappa=\sqrt{2\nu}/l_c$ with $l_c$ the correlation length. Processes with covariance function as in \eqref{def:Matern} have been used in many applications such as machine learning \cite{RW2006} and spatial statistics \cite{S1999}, in particular in geostatistics \cite{DR2007}. Note that for $\nu=1/2$, the covariance function reduces to the exponential
\begin{equation}\label{former}
     {\rm cov}_u(\bx,\bx')=\sigma^2\exp\left(-\frac{|\bx-\bx'|}{l_c}\right)
\end{equation}
while in the limit $\nu\rightarrow\infty$ we get the squared exponential
\begin{equation}\label{latter}
     {\rm cov}_u(\bx,\bx')=\sigma^2\exp\left(-\frac{|\bx-\bx'|^2}{2l_c^2}\right).
\end{equation}
When $D$ is convex, the covariance function \eqref{former} is only Lipschitz continuous and the realizations of the underlying random field $u$ are H\"older continuous with parameter $\alpha<1/2$, namely $u(\cdot,\omega)\in C^{0,\alpha}(\bar D)$ $\mathbb{P}$-almost surely. Instead, we have ${\rm cov}_u\in C^{\infty}(\bar D\times \bar D)$ and $u(\cdot,\omega)\in C^{\infty}(\bar D)$ $\mathbb{P}$-almost surely when the covariance function is given by \eqref{latter}, see \cite{NT2015}.

It was observed in \cite{W1954,W1963} that the stationary (mean zero) solution $u:\mathbb{R}^n\times \Omega\rightarrow\mathbb{R}$ to the SPDE
\begin{equation} \label{def:SPDE_Matern}
     (\kappa^2I-\Delta)^su = w, \quad \mbox{in } \mathbb{R}^n, \,\, \mathbb{P}\mbox{-almost surely},
\end{equation}
has the covariance function \eqref{def:Matern} with smoothness parameter $\nu=2s-n/2$ and marginal variance $\sigma^2=\Gamma(\nu)\Gamma(\nu+n/2)^{-1}(4\pi)^{-n/2}\kappa^{-2\nu}$. Here, $I$ denotes the identity operator, $\Delta$ is the Laplacian operator, and $w$ is Gaussian white noise with unit variance. That is, $w$ is a (so-called generalized) Gaussian random field satisfying
$$
\mathbb{E}\left[(w,\varphi)_{L^2(\mathbb{R}^n)}\right]=0 \qquad \mbox{and} \qquad {\rm Cov}\left((w,\varphi)_{L^2(\mathbb{R}^n)},(w,\psi)_{L^2(\mathbb{R}^n)}\right)=(\varphi,\psi)_{L^2(\mathbb{R}^n)}
$$
for all $\varphi,\psi\in L^2(\mathbb{R}^n)$, where $(\cdot,\cdot)_{L^2(\mathbb{R}^n)}$ stands for the $L^2$ inner product in $\mathbb R^n$, see for instance \cite{LBR2022}.



Instead of using a (truncated) KL expansion to simulate a GRF with covariance function as in \eqref{def:Matern}, we can thus solve the SPDE \eqref{def:SPDE_Matern}. Not only this latter approach does not require the knowledge of the eigenfunctions of the covariance operator, but it also permits several generalizations such as random fields on a bounded domain or non-stationary random fields; we refer to \cite{lindgren2011explicit} for the case $2s\in\mathbb{N}$ and to \cite{bolin2020numerical,BK2020} for the general case $s>n/4$.

\begin{remark} \label{r:recursion}
     In the case $s\ge 1$, namely $s=m+\bar s$ with $m\in\mathbb{N}$ and $\bar s\in[0,1)$, the solution to \eqref{def:SPDE_Matern} can be computed recursively as described in \cite{jansson2021surface}: first solve the $m$ non-fractional SPDEs $\mathcal{L}u^i=u^{i-1}$ for $i=1,\ldots,m$, where $\mathcal{L}:=(\kappa^2I-\Delta)$ and $u^0=w$, and then solve the fractional SPDE $\mathcal{L}^{\bar s}u=u^m$.


\end{remark}

\subsection{Random Fields on Surfaces}
The advantage of the SPDE approach is that it can be straightforwardly generalized to random fields on surfaces \cite{BTMD2020}, while this is not the case when starting from the mean field and covariance function. For instance, replacing the Euclidean distance $|\bx-\bx'|$ in \eqref{def:Matern} by the geodesic distance between $\bx$ and $\bx'$ does not yield an admissible covariance function as it fails to be positive semi-definite, see e.g. \cite{FLH2015}.

In this work, we are thus concerned with the numerical approximation of the following SPDE: find $\widetilde u$ defined on the closed hypersurface $\gamma\subset \mathbb R^n$ ($n=2,3$) such that $\mathbb E[ \widetilde u] = 0$ almost everywhere on $\gamma$, and satisfies
\begin{equation}\label{e:spde}
     (\kappa^2 I- \Delta_\gamma)^s \widetilde u = \widetilde{w},\quad \text{on } \gamma, \,\, \mathbb{P}\mbox{-almost surely},
\end{equation}
where $I$ denotes the identity operator, $\Delta_\gamma$ is the Laplace--Beltrami operator, $\kappa$ is a positive constant, and $\widetilde{w}$ denotes Gaussian white noise with unit variance. The meaning of the fractional power $s>\frac{n-1}{4}$ of the elliptic operator $L:=\kappa^2I-\Delta_\gamma$ will be made precise in the next section.
In statistics, the above equation models GRFs on closed surfaces \cite{lindgren2011explicit}. As above, $s$ indicates the smoothness of the corresponding covariance function while $\kappa$ controls the correlation range \cite{bolin2020numerical}.

In what follows, we shall restrict the fractional power to the range $s\in (\tfrac{n-1}4, 1)$, and we mention that the recursive strategy described in Remark \ref{r:recursion} can be used when $s\ge 1$. We also write $a \lesssim b$ when $a \leq C b$, where $C$ is a positive constant which does not depend on $a$, $b$, or the discretization parameters. We say $a \sim b$ if  $a\lesssim b$ and $b\lesssim a$. We will use the notation $\widetilde .$ for quantities defined on $\gamma$ and discrete functions in space are denoted with capital letters. Finally, all the equations involving random variables or random fields are to be understood in the $\mathbb{P}$-almost surely sense.

\section{Dunford--Taylor Formulation of \texorpdfstring{\eqref{e:spde}}{}}\label{s:model}

The central SPDE~\eqref{e:spde} involves fractional powers of differential operators on surfaces.
In this section, we make these notions precise and justify an integral representation of the solution to \eqref{e:spde}.
The latter is critical for the design and analysis of the numerical algorithm proposed in Section~\ref{s:simulation}.

\subsection{Differential Operators on Surfaces}\label{ss:LB}
We assume that $\gamma$ is a closed and compact hypersurface in $\mathbb{R}^n$ ($n=2,3$) of class $C^3$.
The signed distance function to $\gamma$ is denoted $d:\mathbb R^n \rightarrow \mathbb R$. The regularity of $\gamma$ translates into the regularity of $d$, that is $d \in C^3(\mathcal N)$, where $\mathcal N$ is a tubular neighborhood of $\gamma$ of width depending on the curvatures of $\gamma$, see \cite{delfour2011shapes}.
Note that $\nabla d|_{\gamma}$ is normal to $\gamma$ and thus $\nabla d$ is an extension to $\mathcal N$ of the normal to the surface $\gamma$. With this at hand, one can define the orthogonal projection operator $\mathbf P: \mathcal N \rightarrow \gamma$ by
\begin{equation}\label{e:orthogonal}
     \mathbf P(\bx) := \bx - d(\bx) \nabla d(\bx), \qquad \bx \in \mathcal N,
\end{equation}
which is instrumental to define normal extensions of quantities defined on $\gamma$.
More precisely, for $\widetilde v:\gamma \rightarrow \mathbb R$, the extension is given by $\widetilde v\circ \bP$. By construction, these extensions are constant in the normal direction $\nabla d$.

The surface gradient (or tangential gradient) on $\gamma$ is the tangential component of the regular gradient defined on $\mathcal N$, i.e. $$
 \nabla_\gamma \widetilde v := (I-\nabla d \otimes \nabla d) \nabla v|_{\gamma} : \gamma \rightarrow \mathbb R^n, \quad v := \widetilde v\circ\mathbf P.
$$
The components of the surface gradient $\nabla_\gamma$ are denoted $D_i$, $i=1,\ldots,n$ so that when applied to a vector-valued function $\widetilde \bv =(\widetilde v_1,\ldots,\widetilde v_n): \gamma \rightarrow \mathbb R^n$, the surface gradient reads
$$
     \nabla_\gamma \widetilde \bv := ( D_j \widetilde v_i)_{i,j=1}^n.
$$
Furthermore, the surface divergence and Laplace--Beltrami operator are given for $\widetilde \bv:\gamma \rightarrow \mathbb R^n$ and $\widetilde v: \gamma \rightarrow \mathbb R$ by, respectively,
$$
     \textrm{div}_\gamma \widetilde \bv:= \textrm{trace}(\nabla_\gamma \widetilde \bv), \qquad  \Delta_\gamma \widetilde v := \textrm{div}_\gamma (\nabla_\gamma \widetilde v).
$$

The functional space $L^2(\gamma)$ stands for the space of measurable and square integrable functions on $\gamma$.
It is a Hilbert space with scalar product and norm
$$
     (\widetilde v,\widetilde w)_\gamma := \int_\gamma \widetilde v \ \widetilde w, \qquad  \| \widetilde v \|^2_{L^2(\gamma)}:= (\widetilde v,\widetilde v)_\gamma.
$$
We also define the Sobolev spaces
$$
     H^1(\gamma):= \left\lbrace \widetilde v \in L^2(\gamma) : |\nabla_\gamma \widetilde v| \in L^2(\gamma) \right\rbrace
$$
endowed with the norm
$$
     \| \widetilde v \|_{H^1(\gamma)}^2 := \| \widetilde v \|_{L^2(\gamma)}^2 + \| | \nabla_\gamma \widetilde v |\|_{L^2(\gamma)}^2
$$
and
$$
     H^2(\gamma):= \left\lbrace \widetilde v \in H^1(\gamma) : D_iD_j  \widetilde v \in L^2(\gamma), \quad i,j=1,\ldots,n \right\rbrace
$$
with
$$
     \| \widetilde v \|_{H^2(\gamma)}^2 := \| \widetilde v \|_{H^1(\gamma)}^2 + \sum_{i,j=1}^n \|  D_iD_j \widetilde v \|_{L^2(\gamma)}^2.
$$
Their dual spaces are denoted $H^{-1}(\gamma)$ and $H^{-2}(\gamma)$, respectively, and $\langle \cdot,\cdot \rangle_{H^l(\gamma)}$ stands for the duality product, $l=1,2$.

\subsection{Solution Operator}
We are now in a position to define a shifted Laplace--Beltrami problem associated with the SPDE~\eqref{e:spde}, which formally corresponds to taking $s=1$.
Given $\kappa>0$ and $\widetilde f \in H^{-1}(\gamma)$, we are interested in $\tilde u \in H^1(\gamma)$ satisfying
$$
     \kappa^2 \widetilde u -\Delta_\gamma \widetilde u = \widetilde f
$$
in a weak sense. More precisely, we seek $\widetilde u\in H^1(\gamma)$ satisfying
\begin{equation}\label{e:dirichlet}
     a_\gamma(\widetilde u,\widetilde v):=  \kappa^2\int_\gamma \widetilde u\widetilde v + \int_\gamma \nabla_\gamma \widetilde u \cdot \nabla_\gamma \widetilde  v =  \langle \widetilde f ,\widetilde v \rangle_{H^1(\gamma)}, \Forall \widetilde v\in H^1(\gamma).
\end{equation}
The Lax-Milgram theory implies that the above problem admits a unique solution and we denote by $T: H^{-1}(\gamma)\to H^1(\gamma)$ the solution operator $T\widetilde f:=\widetilde u$.
Because $\gamma$ is of class $C^3$, $T$ is an isomorphism from $L^2(\gamma)$ to $H^2(\gamma)$; see \cite[Theorem~3.3]{Dziuk13} and \cite[Lemma~3]{bonito2020finite}. Whence, its inverse $L := (\kappa^2I-\Delta_\gamma)=T^{-1} $ is defined on $H^2(\gamma)$.

\subsection{Fractional Powers of Elliptic Operators} \label{s:dotted}
The operator $T : H^{-1}(\gamma)\to H^1(\gamma)$ is compact thanks to the compact embedding $H^1(\gamma)\subset H^{-1}(\gamma)$.
Thus, the Fredholm alternative guarantees the existence of an $L^2(\gamma)$ orthonormal basis of eigenfunctions $\{\widetilde\psi_j\}_{j=1}^\infty$ of its inverse $L$ with non-decreasing positive eigenvalues $\kappa^2 =\widetilde\lambda_1 < \widetilde\lambda_2\le \widetilde\lambda_3\le \ldots$. That is, the eigenpair $(\widetilde\lambda_j,\widetilde\psi_j) \in \mathbb R^+ \times H^1(\gamma)$, $j=1,2,\ldots$,  satisfies
\begin{equation} \label{e:eigenpb}
     a_{\gamma}(\widetilde\psi_j,\widetilde v) =  \widetilde\lambda_j\int_{\gamma}\widetilde \psi_j \widetilde v, \Forall \widetilde v\in H^1(\gamma)
\end{equation}
with the eigenfunctions chosen so that
\begin{equation*}
     (\widetilde \psi_j, \widetilde \psi_i)_\gamma = \delta_{ij}, \quad \mbox{for all } i,j=1,2,\ldots.
\end{equation*}

In order to define fractional powers of the shifted Laplace--Beltrami operator $L=\kappa^2I-\Delta_{\gamma}$, we shall define the so-called dotted space as follows. The dotted space $\dH^r(\gamma)$ for $r\geq 0$ is the set of functions $\tilde v \in L^2(\gamma)$ such that
\begin{equation}\label{e:inter-def}
     \|\widetilde v\|_{\dH^r(\gamma)}^2:=\sum_{j=1}^\infty \widetilde \lambda_j^r| (\widetilde v,\widetilde \psi_j)_\gamma|^2 < \infty.
\end{equation}
For negative indices, we define
\[
     \dH^{-r}(\gamma):=    \bigg\{ \bigg\langle \sum_{j=1}^\infty c_j\widetilde\psi_j,\cdot \bigg\rangle_{\dH^r(\gamma)}
     : \ \{c_j\}_{j=1}^\infty \subset \mathbb R, \ \sum_{j=1}^\infty|c_j|^2\widetilde \lambda_j^{-r} < \infty \bigg\}, \quad r > 0.
\]
Here we have identified the duality product between $\dH^{-r}(\gamma)$ and
$\dH^r(\gamma)$ with the $L^2(\gamma)$ scalar product and set
$$
     \bigg\langle \sum_{j=1}^\infty c_j\widetilde\psi_j, \sum_{j=1}^\infty d_j\widetilde\psi_j \bigg\rangle_{\dH^r(\gamma)} := \sum_{j=1}^\infty c_j d_j, \qquad
     \text{for} \qquad \sum_{j=1}^\infty d_j\widetilde\psi_j \in \dH^r(\gamma) .
$$
For $r\in [0,1]$, the dotted spaces $\dH^{4r-2}(\gamma)$ are equivalent to the interpolation spaces $H^{4r-2}(\gamma):=(H^{-2}(\gamma),H^{2}(\gamma))_{r,2}$ obtained by the real interpolation method, see \cite{bonito2015numerical}.
To simplify the presentation, we will frequently use this equivalence, typically to estimate the $H^r(\gamma)$ norm of a function by the $\dH^r(\gamma)$ norm defined in \eqref{e:inter-def}. Moreover, we will simply write the duality pairing with $\langle \cdot,\cdot \rangle$ in what follows.

The fractional powers $L^\beta$, $\beta \in (-1,1)$, of $L$ are defined on $\dH^{r}(\gamma)$ using the eigenpairs of $L$. For  $\widetilde v= \langle\sum_{j=1}^\infty v_j\widetilde\psi_j,\cdot\rangle \in \dH^r(\gamma)$ we define
\begin{equation}\label{e:spectral}
     L^\beta \widetilde v := \bigg\langle\sum_{j=1}^\infty \widetilde\lambda_j^ \beta v_j\widetilde\psi_j,\cdot\bigg\rangle.
\end{equation}
We recall that we have identified the duality product between $\dH^{-r}(\gamma)$ and $\dH^r(\gamma)$ with the $L^2(\gamma)$ scalar product so that when $r\geq 0$ we have
\begin{equation}\label{e:spectral2}
     L^\beta \widetilde v = \sum_{j=1}^\infty \widetilde\lambda_j^\beta v_j\widetilde\psi_j \qquad \textrm{with} \quad v_j := (\widetilde v,\widetilde \psi_j)_\gamma.
\end{equation}
Invoking the Lax-Milgram theory again, we see that $L^\beta:\dH^r(\gamma) \rightarrow \dH^{r-2\beta}(\gamma)$ is an isomorphism and in view of the equivalence of spaces mentioned above,
\begin{equation} \label{eqn:iso}
     L^\beta: H^r(\gamma) \rightarrow H^{r-2\beta}(\gamma) \quad \mbox{is an isomorphism}
\end{equation}
for $\beta\in(-1,1)$ provided that
$$\max\left\{-2,2\beta-2\right\}\le r\le \min\left\{2,2\beta+2\right\}.$$


\subsection{Regularity of the White Noise \texorpdfstring{$\widetilde{w}$}{} and the Solution \texorpdfstring{$\widetilde u$}{}}\label{ss:w-reg}
We have established mapping properties of the fractional powers of $L$ in the previous section.
To justify the SPDE~\eqref{e:spde}, it remains to determine the regularity of the white noise $\widetilde{w}$.

Recall that $(\Omega, \mathcal F,  \mathbb P)$ is a complete probability space. We briefly sketch the proof of \cite[Proposition~2.3]{bolin2020numerical} which guarantees that for $r > \frac{n-1}{2}$, $\widetilde{w}$ satisfies
\begin{equation}\label{e:gn_reg}
     \|\widetilde{w}\|_{L^2(\Omega; H^{-r}(\gamma))}^2=\mathbb E\big[\| \widetilde w \|_{H^{-r}(\gamma)}^2\big]<\infty,
\end{equation}
i.e. $\widetilde{w} \in L^2(\Omega;H^{-r}(\gamma))$. 
Consider the formal KL expansion of $\widetilde w$ with respect to the orthonormal eigenbasis $\{\widetilde\psi_j\}_{j=1}^\infty$ and write
\begin{equation}\label{e:white_noise}
     \widetilde{w} = \big\langle \sum_{j=1}^\infty \xi_j \widetilde\psi_j,\cdot \big\rangle,
\end{equation}
where $\xi_j$ are pairwise independent real-valued standard normally distributed random variables, i.e. $\xi_j\stackrel{i.i.d.}{\sim}\mathcal{N}(0,1)$.
Weyl's law (see e.g. \cite{ivrii2016100}) characterizes the asymptotic behavior of the eigenvalues
\begin{equation}\label{i:weyl}
     \widetilde \lambda_j \sim j^{\frac{2}{n-1}}\quad\text{for } j=1,2,\ldots .
\end{equation}
Whence, for $r>\tfrac{n-1}2$ there holds
\begin{equation}\label{i:w-finite-sum}
     \|\widetilde{w}\|_{L^2(\Omega; H^{-r}(\gamma))}^2
     \sim \sum_{j=1}^{\infty} j^{-\tfrac{2r}{n-1}} \lesssim  \frac{r}{r-(n-1)/2},
\end{equation}
which guarantees that $\widetilde{w} \in L^2(\Omega;H^{-r}(\gamma))$ as anticipated.

We can now link this regularity to the regularity of the solution to the SPDE~\eqref{e:spde} thanks to the isomorphic property \eqref{eqn:iso} of $L^{\beta}$.
Indeed, for $\beta \in (0,1)$ and $\frac{n-1}{2}<r \leq 2$, we have $L^{-\beta}\widetilde w\in H^{2\beta-r}(\gamma)$ $\mathbb{P}$-almost surely.
Moreover, \eqref{e:gn_reg} implies
\begin{equation*}
     \|L^{-\beta}\widetilde{w}\|_{L^2(\Omega; H^{2\beta-r}(\gamma))} \lesssim \|\widetilde{w}\|_{L^2(\Omega;H^{-r}(\gamma))} < \infty
\end{equation*}
so that
\begin{equation}\label{e:reg_sol_generic}
     L^{-\beta}\widetilde w \in L^2(\Omega;H^{2\beta-r}(\gamma))
\end{equation}
and in particular, for $s\in (0,1)$ there exists a unique solution $\widetilde u =L^{-s}\widetilde w$ to \eqref{e:spde} in $L^2(\Omega;H^{2s-r}(\gamma))$.

We end this section by recalling our convention: expressions relating functions from $\Omega$ to a Banach space are always understood $\mathbb P$-almost surely.

\subsection{Dunford--Taylor Integral Representation}
The regularity \eqref{e:reg_sol_generic} implies that restricting $s \in(\tfrac{n-1}{4},1)$ the solution $\widetilde u = L^{-s}\widetilde w$ to our model problem \eqref{e:spde} belongs to $L^2(\Omega;H^{2s-r}(\gamma))\subset L^2(\Omega;L^2(\gamma))$ for any $r\in\left(\tfrac{n-1}{2},2s\right)$.
Under these assumptions, we can proceed as in \cite{bonito2015numerical} and resort to a Dunford--Taylor integral representation of $\widetilde u$ to design a numerical algorithm for its approximation. The next proposition justifies the integral representation \eqref{eqn:integral_repres} by extending \cite[Theorem~2.1]{bonito2015numerical}.

\begin{proposition}\label{p:bochner}
     Let $\widetilde w$ be a Gaussian white noise.  For $\tfrac{n-1}{4}<s<1$, the expression
     \begin{equation}\label{e:integ_piece}
          I(\widetilde w):= \frac{\sin(\pi s)}{\pi}\int_0^\infty \mu^{-s}(\mu I + L)^{-1}\widetilde w \, d\mu
     \end{equation}
     is in $L^2(\Omega; L^2(\gamma))$ and coincides with $L^{-s} \widetilde w$, i.e.
     \[
          L^{-s}\widetilde w  = \frac{\sin(\pi s)}{\pi}\int_0^\infty \mu^{-s}(\mu I + L)^{-1}\widetilde w \, d\mu  \qquad \textrm{in} \quad L^2(\Omega;L^2(\gamma)).
     \]
\end{proposition}
\begin{proof}
     We first show that the mapping
     \begin{equation}\label{e:measurability}
          \Omega \ni \zeta \mapsto \int_0^\infty \mu^{-s} (\mu I+L)^{-1}\widetilde w(\cdot,\zeta) \ d\mu \in L^2(\gamma)
     \end{equation}
     is measurable.
     Recall that \eqref{e:gn_reg} establishes that $\widetilde w\in L^2(\Omega;H^{-r}(\gamma))$ for $r=\tfrac{n-1}{2}+\epsilon$, $\epsilon>0$, and so $\tfrac r 2 < s <1$ for $\epsilon$ sufficiently small. In particular, the mapping $\Omega  \ni \zeta \mapsto \widetilde w(\cdot,\zeta) \in H^{-r}(\gamma)$ is measurable.
     Consequently, the measurability of \eqref{e:measurability} follows from the continuity (or boundedness) of the linear mapping
     \begin{equation}\label{e:map_continuity}
          H^{-r}(\gamma) \ni \widetilde v \mapsto \int_0^\infty \mu^{-s} (\mu I+L)^{-1}\widetilde v \ d\mu \in L^2(\gamma)
     \end{equation}
     discussed now.

     Let $\widetilde v \in H^{-r}(\gamma)$ and set $\widetilde f := L^{-r/2}\widetilde v$ which belongs to $L^2(\gamma)$ thanks to \eqref{eqn:iso}.
     We compute
     $$          \|(\mu I + L)^{-1}\widetilde v\|_{L^2(\gamma)}^2  =
          \|(\mu I + L)^{-1}L^{r/2}\widetilde f\|_{L^2(\gamma)}^2 = \sum_{j=1}^\infty \bigg|\frac{\widetilde \lambda_j^{r/2}}{\mu + \widetilde \lambda_j}\bigg|^2|\tilde f_j|^2,
     $$
     where $\widetilde f_j:=(\widetilde f,\widetilde \psi_j)_{\gamma}$. We estimate the coefficients $\big|\frac{\widetilde \lambda_j^{r/2}}{\mu + \widetilde \lambda_j}\big|$ for $\mu \le 1$ and $\mu >1$ separately.
     When $\mu \le 1$, we recall that $\tfrac r2 < s <1$ and that the eigenvalues are non-decreasingly ordered to deduce that
     \begin{equation}\label{i:r-less-1}
          \frac{\widetilde \lambda_j^{r/2}}{\mu + \widetilde \lambda_j}\lesssim 1,
     \end{equation}
     where the hidden constant depends only on $\widetilde\lambda_1 = \kappa^2$. Instead, when $\mu > 1$, Young's inequality allows us to estimate
     \begin{equation}\label{i:r-larger-1}
          \frac{\widetilde \lambda_j^{r/2}}{\mu + \widetilde \lambda_j} \le \mu^{\tfrac r2 - 1}.
     \end{equation}

     Taking into account these two cases and using the fact that $\|\widetilde f\|_{L^2(\gamma)}\lesssim\|\widetilde v\|_{H^{-r}(\gamma)}$, we find that
     \begin{equation}\label{i:bochner-decay}
          \|(\mu I + L)^{-1}\widetilde v\|_{L^2(\gamma)} \lesssim \|\widetilde v\|_{H^{-r}(\gamma)}
          \left\{
          \begin{aligned}
                & 1,                   &  & \text{when } \mu \le 1, \\
                & \mu^{\tfrac r2 - 1}, &  & \text{when } \mu > 1 .
          \end{aligned}
          \right.
     \end{equation}
     Estimating the integral of this quantity separately for $0<\mu \leq 1$ and $\mu>1$, we deduce that
     \begin{equation}\label{e:bochner_estim}
          \left\| \int_0^\infty \mu^{-s} (\mu I + L)^{-1}\widetilde v d\mu \right\|_{L^2(\gamma)} \leq   \int_0^\infty \mu^{-s} \|(\mu I + L)^{-1}\widetilde v\|_{L^2(\gamma)} d\mu \lesssim \| \widetilde v \|_{H^{-r}(\gamma)}
     \end{equation}
     and the continuity of the mapping \eqref{e:map_continuity} follows. In turns this proves the measurability of the mapping \eqref{e:measurability}. The estimate
     \begin{equation}\label{e:l2_I}
          \| I(\widetilde w) \|_{L^2(\Omega;L^2(\gamma))} \lesssim \| \widetilde w \|_{L^2(\Omega;H^{-r}(\gamma))}
     \end{equation}
     is a direct consequence of \eqref{e:bochner_estim}.

     We now show that the Dunford--Taylor integral $I(\widetilde w)$ coincides with the spectral definition \eqref{e:spectral} of $L^{-s}\widetilde w$, namely $L^{-s}\widetilde w = \sum_{j=1}^\infty \widetilde \lambda_j^{-s}\xi_j \widetilde \psi_j$ with $\widetilde w = \big\langle \sum_{j=1}^\infty \xi_j \widetilde \psi_j,\cdot\big\rangle$.
     We only provide a sketch of the proof since the arguments follow those in \cite{Kato61,Balakrishnan60}, see also \cite{bonito2015numerical,BP15}.
     We start by noting that
     \begin{equation}\label{e:dt-scalar}
          \widetilde \lambda_j^{-s} = \frac{\sin(\pi s)}{\pi}\int_0^\infty \mu^{-s}(\mu +\widetilde\lambda_j)^{-1} \, d\mu
     \end{equation}
     so that for the partial sum $\widetilde w^N := \sum_{j=1}^N \xi_j \widetilde\psi_j$, we have $L^{-s} \widetilde w^N = I(\widetilde w^N)$.
     As a consequence, we have
     \begin{equation*}
          \begin{split}
               &\|L^{-s} \widetilde w - I(\widetilde w)\|_{L^2(\Omega;L^2(\gamma))} \\
               &\qquad \leq \| L^{-s} \widetilde w - L^{-s}\widetilde w^N\|_{L^2(\Omega;L^2(\gamma))} + \|I(\widetilde w^N-\widetilde w)\|_{L^2(\Omega;L^2(\gamma))} \\
               & \qquad \lesssim \| \widetilde w - \widetilde w^N \|_{L^2(\Omega;H^{-r}(\gamma))},
          \end{split}
     \end{equation*}
     where we used the mapping properties \eqref{eqn:iso} of $L^{-s}$, $r/2<s$, and estimate \eqref{e:l2_I} to justify the last inequality.
     To estimate the difference between $\widetilde w$ and $\widetilde w^N$, we resort to Weyl's law \eqref{i:weyl} to obtain
     $$
          \| \widetilde w - \widetilde w^N \|_{L^2(\Omega;H^{-r}(\gamma))}^2 \lesssim \sum_{j=N+1}^\infty j^{-\frac{2r}{n-1}} \to 0 \qquad \textrm{as }N\to \infty.
     $$
     Combining the last two inequalities, we conclude that
     $$
          \|L^{-s} \widetilde w - I(\widetilde w)\|_{L^2(\Omega;L^2(\gamma))} = 0
     $$
     as desired.
\end{proof}

\section{Numerical Scheme}\label{s:simulation}

We propose a numerical scheme based on the integral representation

\begin{equation}\label{e:num_int}
     \widetilde u = L^{-s}\widetilde w = \frac{\sin(\pi s)}{\pi}\int_0^\infty \mu^{-s}(\mu I + L)^{-1}\widetilde w \, d\mu
\end{equation}
of the solution to the SPDE~\eqref{e:spde}.

The numerical approximation of $\widetilde u$ consists of two steps.
Following \cite{herrmann2020multilevel}, we approximate the white noise in a finite element space via projection. Then, the integral representation~\eqref{e:num_int} is approximated using a sinc quadrature combined with a standard finite element method for the approximation of the integrand at each quadrature point.
We refer to \cite{bonito2015numerical,BP15} for the original development and to \cite{bonito2021approximation} for its extension to the fractional Laplace--Beltrami problem on surfaces. The results in \cite{bonito2015numerical,BP15,bonito2021approximation} do not apply in the present stochastic context.

We start with the construction of the finite element space.

\subsection{Finite Element Methods on Discrete Surfaces}\label{ss:fem}

There exist several finite element methods for problems defined on hypersurfaces, see for instance \cite{bonito2020finite} and the references therein. In this work, we consider parametric finite elements \cite{dziuk88} which are defined on an approximated surface $\Gamma$ with the help of the projection $\mathbf P = \bx - d(\bx)\nabla d(\bx)$ mapping $\Gamma$ to $\gamma$ assumed to be $C^3$. Recall that $d$ is the signed distance function and is of class $C^3$, see Section~\ref{ss:LB}. This property guarantees that the geometric inconsistency arising from replacing $\gamma$ by $\Gamma$ does not affect the finite element method convergence rate. For deterministic right-hand sides, this is the subject of Theorem~4.2 in \cite{bonito2021approximation}. The numerical study in \cite{bonito2021approximation} also shows that using a generic continuous piecewise $C^2$ lift \cite{mekchay2011afem,bonito2013afem,bonito2016high} does not affect the convergence rate, noting that such lifting operator only provides the first order convergence for the approximated surface.

Let $\Gamma$ be an $(n-1)$-dimensional polyhedral surface so that the vertices of $\Gamma$ lie on $\gamma$ (see \cite{demlow09} for more general assumptions). The analysis provided below holds for a subdivision made of triangles or quadrilaterals. The former is possibly more standard while the latter is the setting used in the numerical illustrations in Section~\ref{s:numeric} below.  To incorporate both settings simultaneously, we denote by $\mathcal T$ a triangular or quadrilateral subdivision of $\Gamma$. Given a subdivision with triangles (resp. quadrilaterals), let $\hat\tau$ be the unit triangle (resp. square) and denote $\mathcal P$ the set of linear (resp. bi-linear) polynomials defined on $\hat\tau$.

For each $\tau \in \mathcal T$, we set $F_\tau: \widehat \tau \rightarrow \tau$, $F_\tau\in[\mathcal P]^{n-1}$ to be the map from the reference element to the physical element.
We let $c_J:=c_J(\mathcal T)$ be such that
$$
     c_J^{-1} | \mathbf{x} | \leq  | D F_\tau \mathbf{x} | \leq c_J | \mathbf{x}|, \qquad \forall \,\mathbf{x} \in \widehat \tau, \,\, \forall\, \tau \in \mathcal T.
$$
Furthermore, we denote by $h_\tau$, $\tau\in \mathcal T$, the diameter of $\tau$, $h := \max_{\tau \in \mathcal T} h_{\tau}$, and by $c_q:=c_q(\mathcal T)$ the quasi-uniformity constant satisfying
$$
     h:=\max_{\tau\in \mathcal T} h_{\tau} = c_q \min_{\tau\in \mathcal T} h_{\tau}.
$$
The maximum number of elements sharing the same vertex is
\begin{equation}\label{e:uniform-vertices}
     c_v:=c_v(\mathcal T):=\max_{ \mathbf v \textrm{ vertex of } \mathcal T} \# \{ \tau \in \mathcal T \ : \ \mathbf v \textrm{ is a vertex of } \tau\}.
\end{equation}
The constants appearing in the analysis below (but hidden in ``$\lesssim$'') may depend on $c_J$, $c_q$, and $c_v$ while the dependency on the mesh size $h$ will be explicitly given.
We further assume that the geometry of $\gamma$ is sufficiently well approximated by $\Gamma$, i.e. $h$ is small enough, so that $\Gamma \subset \mathcal N$ and thus the lift $\bP$ defined in \eqref{e:orthogonal} restricted to $\Gamma$ is a bijection. To simplify the notation in what follows, we introduce the operator $P:L^1({\gamma})\rightarrow L^1(\Gamma)$ and its inverse $P^{-1}:L^1({\Gamma})\rightarrow L^1(\gamma)$ defined as follows. For $\widetilde v:\gamma \rightarrow \mathbb R$ we set $(P\widetilde v)(\bx):=(\widetilde v\circ\bP)(\bx)$, $\bx\in\Gamma$, while for $v:\Gamma \rightarrow \mathbb R$ we set $(P^{-1}v)(\bx):=(v\circ\bP^{-1})(\bx)$, $\bx\in\gamma$.

\begin{remark} \label{r:meshes}
     It is possible to construct sequences of subdivisions with uniform constants $c_J$, $c_q$, and $c_v$.
     We refer to \cite{bonito2012convergence} for uniform refinements and to \cite{bonito2013afem,bonito2016high} for local refinements by adaptive algorithms. A typical possibility is to start with an initial polyhedral surface $\overline\Gamma_0$ and consider the sequence $\{\Gamma_i:= \mathcal I_i \bP(\overline{\Gamma}_i)\}_{i=1}^\infty$, where $\{\overline\Gamma_i\}$ consists of uniform refinements and $\mathcal I_i$ is the corresponding nodal interpolant.
\end{remark}
The finite element space associated with $\mathcal T$ is denoted $\mathbb V(\mathcal T)$ and reads
$$
     \mathbb V(\mathcal T):= \{ v \in H^1(\Gamma) \ : \ v|_\tau \circ F_\tau  \in \mathcal P, \qquad \forall \tau\in \mathcal T \}.
$$
Its dimension will be denoted by $N:= \textrm{dim}(\mathbb V(\mathcal T))$.

We can now define the discrete operator $L_{\mathcal T}:= T_{\mathcal T}^{-1}: {\mathbb V}(\mathcal T)\to {\mathbb V}(\mathcal T)$ as the discrete counterpart of $L$. Here $T_{\mathcal T}: \mathbb V(\mathcal T)\to \mathbb V(\mathcal T)$ is defined by $T_{\mathcal T}F:= U$ where $U$ is the unique solution to
\begin{equation}\label{e:discrete_T}
     A_\Gamma(U,V):=\kappa^2 \int_\Gamma U  V + \int_\Gamma \nabla_\Gamma  U\cdot
     \nabla_\Gamma V = \int_\Gamma F  V, \quad \text{ for all }
     V\in {\mathbb V}(\mathcal T);
\end{equation}
compare with \eqref{e:dirichlet}.

We also introduce $\sigma : \Gamma\to \mathbb R^+$ to be the ratio between the area element of $\gamma$ and $\Gamma$ associated with the parametrization $\mathbf P$ so that for $\widetilde v \in L^1(\gamma)$ we have
\begin{equation}\label{e:area}
     \int_\gamma \widetilde v = \int_{\Gamma} \sigma P\widetilde v.
\end{equation}
The area ratio $\sigma$ satisfies
\begin{equation}\label{i:sigma}
     \|\sigma\|_{L^\infty(\Gamma)}\lesssim 1, \quad \|1-\sigma\|_{L^\infty(\Gamma)}\lesssim h^2 \quad  \mbox{and} \quad \|1-\sigma^{-1}\|_{L^\infty(\Gamma)}\lesssim h^2,
\end{equation}
see for instance \cite{bonito2020finite}.

\subsection{Lifted Finite Element Method on \texorpdfstring{$\gamma$}{}}\label{s:FEM_gamma}

At several instances, it will be useful to compare quantities defined on $\gamma$ with their finite element approximation on the geometrically consistent lifted finite element space
$$
     \widetilde{\mathbb V}(\mathcal T) := P^{-1} \mathbb V(\mathcal T):=\{ P^{-1}V \ : \ V \in\mathbb V(\mathcal T) \}.
$$
Of course the dimension of $\widetilde{\mathbb V}(\mathcal T)$ is $\textrm{dim}(\widetilde{\mathbb V}(\mathcal T)) = \textrm{dim}({\mathbb V}(\mathcal T))=N$.

The operator $\widetilde L_{\mathcal T}: \widetilde{\mathbb V}(\mathcal T)\to \widetilde{\mathbb V}(\mathcal T)$ is defined as $L_{\mathcal T}$, but on the exact surface $\gamma$. That is, we set $\widetilde L_{\mathcal T}:= \widetilde T_{\mathcal T}^{-1}: \widetilde {\mathbb V}(\mathcal T)\to \widetilde {\mathbb V}(\mathcal T)$ where $\widetilde T_{\mathcal T}: \widetilde {\mathbb V}(\mathcal T)\to \widetilde {\mathbb V}(\mathcal T)$ is defined by $\widetilde T_{\mathcal T}\widetilde F:= \widetilde U$ with $\widetilde U$ uniquely solving
\begin{equation}\label{e:tilde-L}
     \kappa^2 \int_\gamma \widetilde U  \widetilde   V + \int_\gamma \nabla_\gamma  \widetilde U  \cdot
     \nabla_\gamma \widetilde V = \int_\gamma \widetilde F  \widetilde V, \quad \text{ for all }
     \widetilde V\in \widetilde {\mathbb V}(\mathcal T).
\end{equation}

We denote by $\widetilde{\mathbf\Pi}: L^2(\gamma) \rightarrow \widetilde{\mathbb V}(\mathcal T)$ the $L^2(\gamma)$ projection onto $\widetilde{\mathbb V}(\mathcal T)$. It is defined for $\widetilde v \in L^2(\gamma)$ by the relations
\begin{equation}\label{e:L2_conforming}
     \int_\gamma (\widetilde{\mathbf\Pi} \widetilde v - \widetilde v) \widetilde V = 0,\quad\text{ for all } \widetilde V \in \widetilde{\mathbb V}(\mathcal T).
\end{equation}
The Bramble-Hilbert lemma guarantees the approximation property
\begin{equation}\label{e:error_l2_g}
     \|(I-\widetilde{\mathbf\Pi}) \widetilde v\|_{L^2(\gamma)} +h\|(I-\widetilde{\mathbf\Pi}) \widetilde v\|_{H^1(\gamma)} \lesssim h^{t} \|\widetilde v\|_{H^t(\gamma)}
\end{equation}
for all $\widetilde v\in H^t(\gamma)$ with $t\in[0,2]$.

We now turn our attention to the eigenpairs of $\widetilde L_{\mathcal T}$.
We define $\{(\widetilde\Lambda_j^\gamma,\widetilde\Psi_{j}^\gamma)\}_{j=1}^N \subset \mathbb R^+_* \times \widetilde{\mathbb V}(\mathcal T)$ on the exact surface $\gamma$ as satisfying
\begin{equation}\label{e:dicrete_eig}
     \kappa^2\int_\gamma \widetilde \Psi_{j}^\gamma \widetilde V+\int_\gamma \nabla_\gamma \widetilde \Psi_{j}^\gamma\cdot\nabla_\gamma \widetilde V = \widetilde\Lambda_j^\gamma \int_\gamma \widetilde\Psi_{j}^\gamma \widetilde V,
     \quad\text{ for all } \widetilde V\in \widetilde{\mathbb V}(\mathcal T).
\end{equation}
The eigenfunctions are chosen to be $L^2(\gamma)$ orthonormal, i.e.
\begin{equation} \label{e:normalization}
     \int_\gamma \widetilde\Psi_{i}^\gamma\widetilde\Psi_{j}^\gamma = \delta_{ij}, \quad i,j=1,\ldots,N.
\end{equation}

Error estimates for the eigenvalues approximations $|\widetilde\lambda_j - \widetilde\Lambda_j^\gamma|$, $j=1,\ldots,N$, are needed in the proof of the mean square norm error estimate derived in Section~\ref{s:weak} below.
We refer to \cite{babuska1991eigenvalue, boffi2010finite} for standard error estimates on Euclidean domains which typically requires the mesh size to be sufficiently small, where how small depends on the eigenvalue being approximated. Moreover, the constants involved in the error estimates also depend on the magnitude of the eigenvalue being approximated.
We refer to \cite{knyazev2006new} for \emph{cluster robust} error estimates. As a corollary, we obtain the following lemma and note that the provided estimate deteriorates for large eigenvalues, but will nonetheless be sufficient for our analysis.

\begin{lemma}\label{l:eigen-error}
     For $j = 1,\ldots, N$ there holds
     \begin{equation}\label{i:eigen-1}
          0\le \widetilde\Lambda_j^\gamma - \widetilde \lambda_j \lesssim \widetilde\lambda_j \widetilde \Lambda_j^\gamma h^2 ,
     \end{equation}
     where the hidden constant does not depend on $j$ nor on $h$.
\end{lemma}
\begin{proof}
     The desired estimate follows from Theorem~3.1 in \cite{knyazev2006new} which guarantees that for $j=1,\ldots,N$, there holds
     \begin{equation}\label{i:knyazev}
          0\le \frac{\widetilde\Lambda_j^\gamma - \widetilde\lambda_j}{\widetilde\Lambda_j^\gamma} \le
          \sup_{\substack{\widetilde v\in{\rm span}\{\widetilde\psi_l: l\in \{1,\ldots, j\}\}, \\ \|\widetilde v\|_{H^1(\gamma)}=1}}
          \|(I-\widetilde{\mathbf\Pi})\widetilde v\|_{H^1(\gamma)}^2.
     \end{equation}
     Recall that $\widetilde{\mathbf\Pi}$ is the $L^2(\gamma)$ projection onto $\widetilde {\mathbb V}(\mathcal T)$, the lifted finite element space, see \eqref{e:L2_conforming}.
     It remains to estimate the right hand side using the approximation property \eqref{e:error_l2_g} of $\widetilde{\mathbf\Pi}$ to get
     \[
          \|(I-\widetilde{\mathbf\Pi})\widetilde v\|_{H^1(\gamma)}^2 \lesssim h^2 \| \widetilde v \|_{H^2(\gamma)}^2  \lesssim \widetilde\lambda_j h^2,
     \]
     where for the last inequality we used the fact that for $v_l := \int_\gamma \widetilde v\widetilde\psi_l$,
     \[
          \|\widetilde v\|_{H^2(\gamma)}  \lesssim \|L \widetilde v\|_{L^2(\gamma)} = \bigg(\sum_{l=1}^{j} \widetilde\lambda_l^2  v_l^2 \bigg)^{1/2} \lesssim  \widetilde\lambda_j^{1/2} \|\widetilde v\|_{H^1(\gamma)}.
     \]
     Inserting the above two inequalities into \eqref{i:knyazev} yields \eqref{i:eigen-1}.
\end{proof}

\subsection{Approximations of the White Noise \texorpdfstring{$\widetilde{w}$}{}} \label{s:approx_WN}

In \cite{jansson2021surface}, a surface finite element is proposed to approximate \eqref{e:spde} on the sphere using a truncated KL expansion.
For more general surfaces, we follow \cite{herrmann2020multilevel} to construct a finite dimensional approximation of the white noise $\widetilde w$ in the conforming finite element space $\widetilde{\mathbb V}(\mathcal T)$, namely we approximate $\widetilde w = \big\langle \sum_{j=1}^\infty \xi_j \widetilde \psi_j,\cdot\big\rangle$ by
\begin{equation}\label{e:noise-projection}
     \widetilde W := \sum_{j=1}^{\infty}\xi_j\widetilde{\mathbf\Pi} \widetilde \psi_j.
\end{equation}
To show that $\widetilde W \in L^2(\Omega;L^2(\gamma))$ with
$$
     \|\widetilde W\|_{L^2(\Omega; L^2(\gamma))}^2 = N =\text{dim}(\widetilde{\mathbb V}(\mathcal T)),
$$
we follow the proof of \cite[Lemma~2.4]{herrmann2020multilevel}.
Invoking Fubini's theorem and using the orthonormality of the discrete eigenfunctions $\{\widetilde \Psi_j^{\gamma}\}_{j=1}^N$ defined in \eqref{e:dicrete_eig}, we compute
\begin{equation}\label{e:w-finite}
     \begin{aligned}
          \|\widetilde W\|_{L^2(\Omega; L^2(\gamma))}^2 & = \mathbb E\bigg[\sum_{j=1}^\infty \xi_j^2 (\widetilde{\mathbf\Pi} \widetilde \psi_j, \widetilde\psi_j)_\gamma^2 \bigg]
          = \sum_{j=1}^\infty \|\widetilde{\mathbf\Pi} \widetilde\psi_j\|_{L^2(\gamma)}^2                                                                                         \\
                                                        & =\sum_{j=1}^\infty\sum_{i=1}^N (\widetilde{\mathbf\Pi} \widetilde\psi_j, \widetilde\Psi_i^\gamma)_\gamma^2
          =\sum_{i=1}^N\sum_{j=1}^\infty (\widetilde\Psi_i^\gamma, \widetilde\psi_j)_\gamma^2 =\sum_{i=1}^N\|\widetilde\Psi_i^\gamma\|_{L^2(\gamma)}^2                            \\
                                                        & = N .
     \end{aligned}
\end{equation}
We can now use the map $P:L^1(\gamma) \rightarrow L^1(\Gamma)$ to define the approximation to $\widetilde w$ on the discrete surface $\Gamma$, namely
\begin{equation}\label{e:white-noise-app}
     W:= \sigma P\widetilde W \in L^2(\Omega;L^2(\Gamma)),
\end{equation}
where the multiplicative factor $\sigma$ is introduced to avoid geometric inconsistencies when mapping back to $\gamma$; see \eqref{e:area}.


\subsection{Computing the Approximate White Noise \texorpdfstring{$W$}{}} \label{s:alpha}

In this section, we discuss how to compute the approximation $W=\sigma P\widetilde W$ of the white noise $\widetilde w$.
We emphasize that it involves the $L^2(\gamma)$ projection of the continuous eigenfunctions $\widetilde \psi_j$ defined in \eqref{e:eigenpb}, see \eqref{e:noise-projection}.

The construction of the approximate solution $U_k$ given in \eqref{e:sol_ukh} below requires the finite element approximation $U^l$ satisfying \eqref{e:discrete-parametric}. In turn, the latter requires the evaluation of
\begin{equation}\label{def:comp_rhs}
     \alpha_i:= \int_\Gamma W \Phi_i = \int_\Gamma \sigma P\widetilde W \Phi_i, \quad i=1,\ldots,N,
\end{equation}
where $\{\Phi_j\}_{j=1}^N$ denotes the Lagrange nodal basis functions of $\mathbb V(\mathcal T)$.

Using the mapping property \eqref{e:area} and the definition \eqref{e:noise-projection} of $\widetilde W$ together with \eqref{e:L2_conforming}, we write
\[
     \alpha_i = \int_\gamma \widetilde W (P^{-1}{\Phi_i})
     = \sum_{k=1}^\infty \xi_k \int_\gamma \widetilde\psi_k (P^{-1}{\Phi_i}),
\]
where we recall that $\xi_k\stackrel{i.i.d.}{\sim}\mathcal{N}(0,1)$.
This implies that $\mathbb{E}[\alpha_i]=0$, $i=1,\ldots,N$, and
\begin{equation} \label{def:weightedMass}
     \begin{aligned}
          \mathbb E[\alpha_i \alpha_j] & = \sum_{k=1}^\infty \left(\int_\gamma \widetilde\psi_k  (P^{-1}{\Phi_i})\right) \left( \int_\gamma \widetilde\psi_k (P^{-1}{\Phi_j}) \right) \\
                                       & = \int_\gamma (P^{-1}{\Phi_i})  (P^{-1}{\Phi_j})
          = \int_\Gamma \sigma\Phi_i \Phi_j =: M_{ij}, \quad i,j=1,\ldots,N,
     \end{aligned}
\end{equation}
where we used the orthonormality of the eigenfunctions $\{\widetilde\psi_j\}_{j=1}^\infty$.
Hence, the vector $\boldsymbol{\alpha}:= (\alpha_1,\ldots,\alpha_N)^T$ satisfies $\boldsymbol{\alpha}\sim \mathcal N(\mathbf 0,M)$, where $M:=(M_{ij})_{i,j=1}^N$ is the $\sigma$-weighted mass matrix associated to the Lagrange finite element basis $\{\Phi_j\}_{j=1}^N$.
As a consequence, if $G$ denotes a Cholesky factor of $M$, i.e. $GG^T=M$, the above reasoning indicates that $G\bsz$ with $\bsz\sim \mathcal N(\mathbf 0,I_{N\times N})$ follows the same distribution as ${\boldsymbol \alpha}$ and could thus be used in place of $\boldsymbol \alpha$.

For the analysis below, we shall need a different representation of $\widetilde W$ based on the discrete eigenfunctions $\{\widetilde\Psi_j^\gamma\}_{j=1}^N$ defined on the exact surface $\gamma$, see Section~\ref{s:FEM_gamma}. Namely, we set
\begin{equation}\label{e:noise_eigenfcts}
     \widetilde W^\Psi := \sum_{j=1}^N \xi_j \widetilde\Psi_j^\gamma,
\end{equation}
where $\xi_j\stackrel{i.i.d.}{\sim}\mathcal{N}(0,1)$; compare with \eqref{e:noise-projection}.
The corresponding random data vector is given by
\[
     {\boldsymbol \alpha}^\Psi :=(\alpha_i^\Psi)_{i=1}^N:=\bigg(\int_\Gamma\sigma P\widetilde W^\Psi\Phi_i\bigg)_{i=1}^N.
\]

In fact, ${\boldsymbol \alpha}^\Psi$ follows the same distribution as ${\boldsymbol \alpha}$. To verify that this is true, let $R\in \mathbb R^{N\times N}$ be the matrix with entries
\[
     R_{ij} := \int_\Gamma \sigma (P\widetilde\Psi_i^\gamma) \Phi_j,
\]
so that $\alpha_i^\Psi = (R^T\bsxi)_i$ for $i=1,\ldots,N$, where $\bsxi:= (\xi_1,\ldots,\xi_N)^T$. Recalling that $M$ is the $\sigma$-weighted mass matrix with entries $M_{ij}=\int_\Gamma \sigma\Phi_i\Phi_j$, using the orthonormality property \eqref{e:normalization} we infer that $M = R^T R$. Thus,
\[
     \mathbb E[{\boldsymbol \alpha}^\Psi({\boldsymbol \alpha}^\Psi)^T] = \mathbb E[R^T\bsxi\bsxi^T R] = M .
\]
This shows that ${\boldsymbol \alpha}^\Psi,{\boldsymbol \alpha} \sim \mathcal N(\mathbf 0, M)$.

Thanks to the eigenvalues approximation estimate (Lemma~\ref{l:eigen-error}), we deduce estimates for $\| \widetilde L^{-r/2}_{\mathcal T}\widetilde W^\Psi\|_{L^2(\Omega; L^2(\gamma))}$ when $r\in(\tfrac{n-1}2,2s)$; compare with the estimates \eqref{i:w-finite-sum} for $\widetilde w$.  This is the object of the following lemma.
\begin{lemma}\label{l:negative-boundedness}
     For $r\in (\tfrac{n-1}2, 2s)$, we have
     \begin{equation}\label{i:w-eigen-finite}
          \|\widetilde L^{-r/2}_{\mathcal T}\widetilde W^\Psi\|_{L^2(\Omega;L^2(\gamma))}^2 \lesssim \frac{r}{r-(n-1)/2} .
     \end{equation}
\end{lemma}
\begin{proof}
     Using the expansion in the eigenfunction basis $\{ \widetilde \Psi_i^{\gamma} \}_{i=1}^N$ of $\widetilde L_{\mathcal T}$ (discrete operator defined on the lifted space $\widetilde{\mathbb V}(\mathcal T)$, see \eqref{e:tilde-L}), we find that
     \[
          \|\widetilde L_{\mathcal T}^{-r/2} \widetilde W^\Psi\|_{L^2(\Omega;L^2(\gamma))}^2
          = \sum_{j=1}^N (\widetilde\Lambda_j^\gamma)^{-r}
          = \sum_{j=1}^N [(\widetilde\Lambda_j^\gamma)^{-r} - \widetilde\lambda_j^{-r}]
          + \sum_{j=1}^N \widetilde\lambda_j^{-r}.
     \]
     The approximation estimates for the eigenvalues (Lemma~\ref{l:eigen-error}) yield
     \[
          |(\widetilde\Lambda_j^\gamma)^{-r} - \widetilde\lambda_j^{-r}|
          \le r\widetilde\lambda_j^{-r-1}|\widetilde\Lambda_j^\gamma - \widetilde\lambda_j|
          \lesssim h^2 \widetilde\lambda_j^{-r}\widetilde\Lambda_j^\gamma
          \lesssim \widetilde\lambda_j^{-r},
     \]
     where for the last inequality we used the fact that $\widetilde\Lambda_j^\gamma\le \widetilde \Lambda_N^\gamma \lesssim h^{-2}$.
     Therefore, we obtain
     $$
          \|\widetilde L_{\mathcal T}^{-r/2} \widetilde W^\Psi\|_{L^2(\Omega;L^2(\gamma))}^2 \lesssim \sum_{j=1}^N \widetilde\lambda_j^{-r}.
     $$
     It remains to take advantage of the decay of the eigenvalues \eqref{i:weyl}, see also \eqref{i:w-finite-sum}, to derive the desired estimate.
\end{proof}

\subsection{Finite Element Approximation for \texorpdfstring{$L^{-s}$}{}}

The numerical scheme advocated here for approximating \eqref{e:num_int} follows the original work in \cite{BLP17}.
Although the scheme is conceptually the same, its analysis must be extended, see Section~\ref{s:analysis}.

With the change of variable $y=\ln(\mu)$, so that $\mu = e^y$, we rewrite \eqref{e:num_int} as
$$
     \widetilde u =L^{-s} \widetilde  w  = \frac{\sin(\pi s)}{\pi} \int_{-\infty}^\infty e^{(1-s)y}(e^y I+L)^{-1} \widetilde w\, dy.
$$
The improper integral in $y$ is approximated using a sinc quadrature with spacing $k>0$, i.e.
\begin{equation}\label{e:approx_quad}
     \widetilde u  \approx  \widetilde u_k:=\cQ^{-s}_k(L)\widetilde w := \frac{k\sin(\pi s)}{\pi} \sum_{l=-\mathtt M}^{\mathtt N} e^{(1-s) y_l} (e^{y_l}I+L)^{-1} \widetilde w,
\end{equation}
where
\begin{equation}\label{e:choose}
     \mathtt N:=\bigg\lceil \frac {2\pi^2}{(s-(n-1)/4) k^2}\bigg\rceil,
     \qquad
     \mathtt M:=\bigg \lceil \frac {\pi^2}{(1-s) k^2}\bigg \rceil
\end{equation}
and $y_l := k l$, $l = -\mathtt M,\ldots,\mathtt N$. This particular choice of $\mathtt N$ and $\mathtt M$ is motivated by the analysis provided in Section~\ref{s:analysis}.

The sinc quadrature approximation is stable as indicated by the following lemma.
\begin{lemma}\label{l:sinc-stability}
     For all $\widetilde v \in L^2(\Omega;H^{-r}(\gamma))$, there holds
     \begin{equation}\label{e:sinc-stab-cont}
          \| \cQ_{k}^{-s}(L) \widetilde v \|_{L^2(\Omega;L^2(\gamma))} \lesssim \| \widetilde v \|_{L^2(\Omega;H^{-r}(\gamma))}.
     \end{equation}
\end{lemma}
In the interest of being concise, we do not provide a proof of this result since estimate \eqref{e:sinc-stab-cont} follows from similar arguments leading to \eqref{e:l2_I} (for a finite sum instead of the integral). 

The numerical approximation of \eqref{e:approx_quad} is based on \cite{bonito2021approximation}.
We recall that an to guarantee optimal convergence rate for the deterministic PDE $(-\Delta_\gamma) \widetilde u = \widetilde f$, the right hand side used in the numerical algorithm is chosen to be $\sigma P \widetilde f$. In the present context, this suggests the use of $\sigma P\widetilde W\in L^2(\Omega;L^2(\Gamma))$ as introduced in Section~\ref{s:approx_WN}, where $\widetilde W$ is defined in \eqref{e:noise-projection}. Consequently, the approximation $U_{k} \in L^2(\Omega;\mathbb V(\mathcal T))$ to $\widetilde u$ is defined as
\begin{equation}\label{e:sol_ukh}
     U_{k}:= \cQ_{k}^{-s}(L_{\mathcal T}) (\sigma P\widetilde W) := \frac{k\sin(\pi s)}{\pi} \sum_{l=-\mathtt M}^{\mathtt N}e^{(1-s) y_l} U^l,
\end{equation}
where $U^l:=(e^{y_l}I+L_{\mathcal T})^{-1}(\sigma P\widetilde {W}) \in L^2(\Omega;\mathbb V(\mathcal T))$ approximates $\widetilde u^{l}:= \widetilde u^l (\widetilde w) := (e^{y_l}I+L)^{-1} \widetilde w$, namely it satisfies
\begin{equation}\label{e:discrete-parametric}
     (e^{y_l} + \kappa^2)\int_{\Gamma} U^l V  +
     \int_{\Gamma} \nabla_{\Gamma} U^l \cdot\nabla_{\Gamma}V
     = \int_{\Gamma} \sigma (P\widetilde W) V,
     \quad\text{ for all } V\in \mathbb V(\mathcal T).
\end{equation}
In other words, recalling the discussion in Section~\ref{s:alpha} for the computation of the approximate white noise, we have $U^l({\bf x},\omega) = \sum_{j=1}^N u_j^l(\omega)\Phi_j({\bf x})$ where the random vector ${\bf u}^l=(u_j^l)_{j=1}^N:\Omega\rightarrow\mathbb{R}^N$ solves the linear system
\begin{equation} \label{def:LS_random_l}
    A^l{\bf u}^l = G{\bf z}.
\end{equation}
Here ${\bf z}\sim \mathcal N(\mathbf 0,I_{N\times N})$, $G$ a Cholesky factor of the weighted mass matrix $M$ defined in \eqref{def:weightedMass}, and $A^l=(a_{ij}^l)_{i,j=1}^N\in\mathbb{R}^{N\times N}$ is given by
$$
a_{ij}^l:= (e^{y_l} + \kappa^2)\int_{\Gamma} \Phi_j\Phi_i  +
     \int_{\Gamma} \nabla_{\Gamma} \Phi_j \cdot\nabla_{\Gamma}\Phi_i,\quad i,j=1,\ldots,N.
$$


Regarding the finite element error, we recall that Theorem~4.2 of \cite{bonito2021approximation} guarantees that for $\widetilde f \in L^2(\gamma)$ with vanishing mean value and $L=(-\Delta_\gamma)$, there holds
\[
     \|P\mathcal Q_k^{-s}(L)\widetilde f - \mathcal Q_k^{-s}(L_\mathcal T)(\sigma P\widetilde f)\|_{L^2(\Gamma)} \lesssim h^{2s}\ln(h^{-1}) \|\widetilde f\|_{L^2(\gamma)}
\]
provided $h < e^{-1}$.
This result readily extends to our current setting where the operator $L$ is $\kappa^2 I - \Delta_{\gamma}$ and the right hand side has a stochastic component. More precisely, we have
\begin{equation}\label{e:existingFEM}
     \|P\mathcal Q_k^{-s}(L)\widetilde W - \mathcal Q_k^{-s}(L_\mathcal T)(\sigma P\widetilde W)\|_{L^2(\Omega; L^2(\Gamma))} \lesssim h^{2s}\ln(h^{-1}) \|\widetilde W\|_{L^2(\Omega;L^2(\gamma))},
\end{equation}
where now the hidden constant depends on $\kappa$. A similar estimate holds for the finite element method defined on the lifted finite element space $\widetilde{\mathbb V}(\mathcal T)$ (see Section~\ref{s:FEM_gamma}), namely for the error between $\mathcal Q_k^{-s}(L)\widetilde W$ and $\mathcal Q_k^{-s}(\widetilde L_\mathcal T)\widetilde W$.

\subsection{Lifted Finite Element Method on \texorpdfstring{$\gamma$}{}}\label{ss:FEM_gamma}
As mentioned above, the approximation $\widetilde W^{\Psi}$ of the white noise defined in \eqref{e:noise_eigenfcts} will be needed for the analysis of the method. More precisely, the mean square norm error estimate derived in Section~\ref{s:weak} below involves the approximations $\mathcal Q_k^{-s}(\widetilde L_{\mathcal T}) \widetilde W^\Psi$ and $\mathcal Q_k^{-s}(L_{\mathcal T}) (\sigma P\widetilde W^\Psi)$ obtained with the finite element methods defined on $\widetilde{\mathbb V}(\mathcal T)$ and $\mathbb V(\mathcal T)$, respectively. We end this section by stating important results for these various approximations.


First, recalling that the random data vectors ${\boldsymbol \alpha}$ and ${\boldsymbol \alpha}^{\Psi}$ based on $\widetilde W$ and $\widetilde W^{\Psi}$, respectively, follow the same Gaussian distribution, we have
\begin{equation}\label{e:weak-norm-equal}
     \|U_k\|_{L^2(\Omega;L^2(\Gamma))} =  \|\cQ_{k}^{-s}(L_{\mathcal T}) (\sigma P\widetilde W^{\Psi})\|_{L^2(\Omega;L^2(\Gamma))}.
\end{equation}
Moreover, using the expansion in the eigenfunction basis $\{ \widetilde \Psi_i^{\gamma} \}_{i=1}^N$ of $\widetilde L_{\mathcal T}$, the mean square norm of $\mathcal Q_k^{-s}(\widetilde L_{\mathcal T}) \widetilde W^\Psi$ is given by
\begin{equation}\label{e:mean_square_UkPsi}
     \|\mathcal Q_k^{-s}(\widetilde L_{\mathcal T}) \widetilde W^\Psi\|_{L^2(\Omega;L^2(\gamma))}^2
     = \sum_{j=1}^N  \mathcal Q_k^{-s}(\widetilde\Lambda_j^\gamma)^2 .
\end{equation}
Finally, we can show that the strong mean square error between $P\mathcal Q_k^{-s}(\widetilde L_{\mathcal T}) \widetilde W^\Psi$ and $\mathcal Q_k^{-s}(L_{\mathcal T}) (\sigma P\widetilde W^\Psi)$ is controlled by the geometric error. The analysis follows from the finite element error analysis in Section~4 of \cite{bonito2021approximation} and relies on the geometric error estimate for the weak formulation as well as regularity properties for the discrete operators, see Appendix~\ref{sec:appendixB} for the complete proof.
\begin{lemma}\label{l:frac-geo-error}
     There holds
     \[
          \|P\mathcal Q_k^{-s}(\widetilde L_{\mathcal T}) \widetilde W^\Psi-\mathcal Q_k^{-s}(L_{\mathcal T}) (\sigma P\widetilde W^\Psi)\|_{L^2(\Omega;L^2(\Gamma))} \le C(h) h^2,
     \]
     where $C(h)\lesssim 1$ when $n=2$ and $C(h)\lesssim \ln(h^{-1})$ when $n=3$ provided that $h<e^{-1}$.
\end{lemma}

\section{Strong Error Estimate}\label{s:analysis}

In this section, we provide a mean square error estimate for the numerical approximation $U_k$ in \eqref{e:sol_ukh} to the solution $\widetilde u$ of the model problem \eqref{e:spde}.
We start with the convergence of the sinc quadrature scheme for the continuous operator, i.e. the discrepancy between $\widetilde u=L^{-s}\widetilde w$ in \eqref{e:num_int} and $\widetilde u_k = \mathcal Q_k^{-s}(L)\widetilde{w}$ expressed in \eqref{e:approx_quad}. Then, we address the finite element error between $\widetilde u_k$ and $\widetilde U_k = P^{-1}U_k$ on $\gamma$.
We note that by considering the sinc quadrature error first, our final estimate for the error $\widetilde u-\widetilde U_k$ improves the results of \cite{bolin2020numerical,herrmann2020multilevel}.

\subsection{Analysis of the Sinc Approximation}

The following result establishes the exponential convergence of the sinc quadrature approximation $\widetilde u_k=\mathcal Q_k^{-s}(L)\widetilde{w}$ to $\widetilde u$ in the strong mean square norm $\|.\|_{L^2(\Omega;L^2(\gamma))}$.
\begin{proposition}\label{p:sinc}
     Assume that $\frac{n-1}{4}<s<1$ and set $r=\frac{n-1}{4}+s$ so that $\widetilde w \in L^2(\Omega;H^{-r}(\gamma))$. Let $\widetilde u$ be the solution to \eqref{e:spde} and let $\widetilde u_k = \mathcal Q_k^{-s}(L)\widetilde{w}$ be defined by \eqref{e:approx_quad} with a sinc quadrature step $k>0$. There holds
     \[
          \|\widetilde u-\widetilde u_k\|_{L^2(\Omega;L^2(\gamma))} \lesssim e^{-\frac{\pi^2}k} \| \widetilde w\|_{L^2(\Omega;H^{-r}(\gamma))}.
     \]
\end{proposition}
\begin{proof}
     We only sketch a proof of this results since it follows the same arguments used in Theorem~3.2 of \cite{BLP17} except that the data $\widetilde w$ belongs to $L^2(\Omega;H^{-r}(\gamma))$ rather than $L^2(\gamma)$.

     Note that the specific choice of $r$ is made so that not only $\widetilde w \in L^2(\Omega;H^{-r}(\gamma))$ but also $\widetilde u \in L^2(\Omega;L^2(\gamma))$, see Section~\ref{ss:w-reg}.
     In particular, we have $\widetilde v := L^{-\frac r2} \widetilde w \in L^2(\Omega;L^2(\gamma))$ satisfies $\widetilde w = L^{\frac r2} \widetilde v$ with $\| \widetilde v\|_{L^2(\Omega;L^2(\gamma))} \lesssim \| \widetilde w \|_{L^2(\Omega;H^{-r}(\gamma))}$. Hence, using the eigenfunction expansion of $\widetilde v$ it suffices to investigate the quadrature error in the scalar case, namely
     \[
          e(\widetilde\lambda) := \bigg|\int_{-\infty}^\infty g_{\widetilde\lambda}(y) \,dy - k\sum_{l=-\mathtt M}^{\mathtt N} g_{\widetilde\lambda}(l k)\bigg|
     \]
     for all $\widetilde \lambda \ge \widetilde\lambda_1$. Here, the integrand
     \[
         g_{\widetilde\lambda}(z) :=\exp((1-s)z) (\exp(z) + \widetilde \lambda)^{-1}\widetilde\lambda^{\frac r2}
     \]
     arises from the relation
     $$
          e^{(1-s)y}(e^yI+L)^{-1} L^{\frac r2}  \widetilde \psi_j  = g_{\widetilde\lambda_j}(y) \widetilde \psi_j.
     $$
     It is analytic in the complex strip $\{z\in \mathbb C : |\Im z|\le \tfrac\pi 2\}$ and exhibits the following decay property (cf. \cite[Lemma~3.1]{BLP17})
     \begin{equation}\label{i:decay}
          |g_{\widetilde\lambda}(z)| \lesssim
          \left\{
          \begin{aligned}
                & \exp(-(s - \tfrac r2)\Re(z)), &  & \text{for } \Re(z)>0,     \\
                & \exp((1-s) \Re(z)),           &  & \text{for } \Re(z)\le 0 .
          \end{aligned}
          \right .
     \end{equation}
     Hence, the fundamental theorem for the sinc approximation (Theorem~2.20 in \cite{lund1992sinc}) ensures the error estimate
     \begin{equation}\label{i:sclar-exp}
          \begin{aligned}
               |e(\widetilde \lambda)| & \lesssim  \sinh(\pi^2/(2k))^{-1}e^{-\pi^2/(2k)} +  \textcolor{black}{\frac{1}{s-r/2}}e^{-(s-r/2)\mathtt Nk} + \textcolor{black}{\frac{1}{1-s}}e^{-(1-s)\mathtt Mk} \\
                                       & \lesssim e^{-\pi^2/k} +  e^{-[s-(n-1)/4]\mathtt Nk/2} + e^{-(1-s)\mathtt Mk},
          \end{aligned}
     \end{equation}
     where we used that $r=\tfrac{n-1}{4}+s$. The desired estimate follows with the choice of  $\mathtt M$ and $\mathtt N$ in \eqref{e:choose}, which ensures that the three exponents on the right hand side above are balanced
     \[
          e^{-[s-(n-1)/4]\mathtt Nk/2} \approx e^{-\pi^2/k} \approx e^{-(1-s)\mathtt Mk} .
     \]
     Hence, in view of \eqref{e:gn_reg} for a fixed $r$,
     \[
          \|\widetilde u-\widetilde u_k\|_{L^2(\Omega;L^2(\gamma))} \lesssim \big(\max_{\widetilde\lambda\ge\widetilde\lambda_1}|e(\widetilde\lambda)|\big)\|\widetilde{w}\|_{L^2(\Omega; H^{-r}(\gamma))}\lesssim e^{-\pi^2/k}    \|\widetilde{w}\|_{L^2(\Omega; H^{-r}(\gamma))} .
     \]
     The proof is complete.
\end{proof}
\begin{remark}
     Notice that the result of Proposition~\ref{p:sinc} holds for any choice of
     $$
          r\in \left( \frac{n-1}{2},2s \right)
     $$
     which guarantees that both $\widetilde w \in L^2(\Omega;H^{-r}(\gamma))$ and $\widetilde u \in L^2(\Omega;L^2(\gamma))$, see Section~\ref{ss:w-reg}.
     Furthermore, recall that according to \eqref{i:w-finite-sum} we have $\| \widetilde w \|_{L^2(\Omega; H^{-r}(\gamma))} \to +\infty$ as $r\to (n-1)/2$. This means that the error estimate provided by  Proposition~\ref{p:sinc} deteriorates as $s\to (n-1)/4$.
\end{remark}

\subsection{Strong Error Estimate}

The strong convergence of the finite element algorithm is discussed in this section. This result, together with the mean square norm convergence discussed in the next section, are the main results of this work.

\begin{theorem}[strong convergence]\label{t:strong}
     Assume that $\tfrac{n-1}4<s<1$. Let $\widetilde u$ be the solution to \eqref{e:spde} and let $U_k$ be the discrete approximation defined in \eqref{e:sol_ukh}. Then there holds
     \[
          \|\widetilde u - P^{-1}U_k\|_{L^2(\Omega;L^2(\gamma))} \lesssim \textcolor{black}{(\ln(h^{-1}))^{3/2}}h^{2s-(n-1)/2} + e^{-\pi^2/k}
     \]
     provided $h < e^{-1}$.
\end{theorem}
\begin{proof}
     We split the error into the sinc quadrature error, the white noise approximation error, and the finite element error
     \begin{equation}\label{i:triangle}
          \begin{split}
               &\|\widetilde u - P^{-1} U_k\|_{L^2(\Omega;L^2(\gamma))} \\
               & \qquad \leq
               \|\widetilde u - \widetilde u_k\|_{L^2(\Omega;L^2(\gamma))}
               + \|\widetilde u_k - \widetilde u_k^N\|_{L^2(\Omega;L^2(\gamma))}
               + \|\widetilde u_k^N - P^{-1}U_k\|_{L^2(\Omega;L^2(\gamma))},
          \end{split}
     \end{equation}
     where $\widetilde u_k^N := \mathcal Q_k^{-s}(L)\widetilde W$ with $\widetilde W$ defined in \eqref{e:noise-projection}.

     The quadrature error is already estimated in Proposition~\ref{p:sinc} and reads
     $$
          \|\widetilde u - \widetilde u_k\|_{L^2(\Omega;L^2(\gamma))}\lesssim e^{-\frac{\pi^2}k} \| \widetilde w\|_{L^2(\Omega;H^{-r}(\gamma))}, \qquad r=\frac{n-1}{4}+s.
     $$

     For the white noise approximation error, we take advantage of the eigenfunction expansion and the stability of the sinc quadrature (Lemma~\ref{l:sinc-stability}, \textcolor{black}{see also the first inequality of \eqref{i:sclar-exp}) to infer that}
\textcolor{black}{
     \begin{equation*}
          \begin{aligned}
               \|\widetilde u_k-\widetilde u_k^N&\|_{L^2(\Omega; L^2(\gamma))}^2  = \|\mathcal Q_k^{-s}(L)(\widetilde w - \widetilde W)\|_{L^2(\Omega; L^2(\gamma))}^2 \\
                & \lesssim \epsilon^{-2}\|L^{-(s-\epsilon/2)}(\widetilde w-\widetilde W)\|_{L^2(\Omega; L^2(\gamma))}^2
               \lesssim \epsilon^{-2}\sum_{j=1}^\infty \widetilde\lambda_j^{-2s+\epsilon} \|(I-\widetilde{\mathbf\Pi})\widetilde\psi_j\|_{L^2(\gamma)}^2,
          \end{aligned}
     \end{equation*}
     where $0<\epsilon<2s-(n-1)/2$ will be chosen below.} The approximation property \eqref{e:error_l2_g} of $\widetilde{\mathbf\Pi}$ together with the property
     \[
          \|\widetilde\psi_j\|_{H^t(\gamma)} \lesssim \|L^{t/2}\widetilde\psi_j\|_{L^2(\gamma)}\lesssim  \widetilde \lambda_j^{t/2}\|\widetilde\psi_j\|_{L^2(\gamma)} = \widetilde \lambda_j^{t/2},
     \]
     both with \textcolor{black}{$t=2s-(n-1)/2-\widetilde\epsilon$, where $\epsilon<\widetilde\epsilon<2s-(n-1)/2$ so that $t\in (0,2)$},
     yield
     $$
          \|\widetilde u_k-\widetilde u_k^N\|_{L^2(\Omega; L^2(\gamma))}^2 \lesssim \textcolor{black}{\epsilon^{-2}h^{4s-(n-1)-2\widetilde\epsilon} \sum_{j=1}^\infty \widetilde\lambda_j^{-(n-1)/2-(\widetilde\epsilon-\epsilon)}}.
     $$
     Since $h < e^{-1}$, we choose \textcolor{black}{$\widetilde\epsilon = (2s-(n-1)/2)/\ln(h^{-1}) \in (0,2s-(n-1)/2)$ and $\epsilon=\widetilde\epsilon/2$} and deduce, using Weyl's law \eqref{i:weyl}, that
     \begin{equation}\label{est:finite_dim}
          \begin{aligned}
               \|\widetilde u_k-\widetilde u_k^N\|_{L^2(\Omega; L^2(\gamma))}^2 & \lesssim  \textcolor{black}{\epsilon^{-2}h^{4s-(n-1)-2\widetilde\epsilon}} \sum_{j=1}^\infty j^{-1-2(\widetilde\epsilon-\epsilon)/(n-1)}           \\
               & \lesssim \textcolor{black}{\epsilon^{-3}  h^{4s-(n-1)-2\widetilde\epsilon}} \lesssim \textcolor{black}{(\ln(h^{-1}))^3 h^{4s-(n-1)}}.
          \end{aligned}
     \end{equation}

     It remains to estimate the finite element error.
     Note that in view of the change of variable formula \eqref{e:area} and the estimates on the area ratio \eqref{i:sigma}, we have
     \begin{equation} \label{est:gamma_Gamma}
          \begin{aligned}
               \|\widetilde u_k^N - P^{-1} U_k\|_{L^2(\Omega;L^2(\gamma))}
                & = \|\sigma^{1/2}(P\widetilde u_k^N - U_k)\|_{L^2(\Omega;L^2(\Gamma))}                                                       \\
                & \lesssim \|P\widetilde u_k^N - U_k\|_{L^2(\Omega;L^2(\Gamma))}                                                              \\
                & =\|P\mathcal Q_k^{-s}(L)\widetilde W - \mathcal Q_k^{-s}(L_\mathcal T)(\sigma P\widetilde W)\|_{L^2(\Omega; L^2(\Gamma))} .
          \end{aligned}
     \end{equation}
     Whence, \eqref{e:existingFEM} together with \eqref{e:w-finite} yield
     $$
          \|\widetilde u_k^N - P^{-1} U_k\|_{L^2(\Omega;L^2(\gamma))} \lesssim h^{2s} \ln(h^{-1}) N^{\frac 1 2}\lesssim \ln(h^{-1})h^{2s-(n-1)/2}.
     $$

     The desired estimate follows upon gathering the three error estimates.
\end{proof}

\begin{remark}\label{r:sol-bounded}
     Note that from Theorem~\ref{t:strong}, we can deduce the following stability of the sinc quadrature applied to $L_{\mathcal T}$
     \begin{equation}\label{i:discrete-solution-bound-A}
\begin{split}
          &\|\mathcal Q_k^{-s}(L_{\mathcal T}) (\sigma P\widetilde W^\Psi)\|_{L^2(\Omega;L^2(\Gamma))} \\
          & \qquad \lesssim \| \widetilde u \|_{L^2(\Omega;L^2(\gamma))} + \textcolor{black}{(\ln(h^{-1}))^{3/2}}h^{2s-(n-1)/2} + e^{-\pi^2/k}
          \end{split}
     \end{equation}
     and thanks to Lemma~\ref{l:frac-geo-error}, the same hold for $\widetilde L_{\mathcal T}$
     \begin{equation}\label{i:discrete-solution-bound-B}
          \|P\mathcal Q_k^{-s}(\widetilde L_{\mathcal T}) (\widetilde W^\Psi)\|_{L^2(\Omega;L^2(\Gamma))}  \lesssim \| \widetilde u \|_{L^2(\Omega;L^2(\gamma))} + \textcolor{black}{(\ln(h^{-1}))^{3/2}}h^{2s-(n-1)/2} + e^{-\pi^2/k}.
     \end{equation}
\end{remark}

           
\begin{remark}\label{r:sinc-independent}
In view of Theorem~\ref{t:strong}, we have improved the strong error estimate from \cite[Theorem~2.10]{bolin2020numerical} and \cite[Proposition~2.2]{herrmann2020multilevel} by showing that the error due to the sinc quadrature is independent of the mesh size $h$. This can be verified numerically. According to the proof of Proposition~\ref{p:sinc}, it suffices to check the error $\max_{\widetilde\lambda\ge\widetilde\lambda_1}|e(\widetilde\lambda)|$. In Figure~\ref{fig:sinc}, we report the error $|e(\widetilde\lambda)|$, $\widetilde\lambda\in [2,10^7]$, when the fractional power is set to $s=0.75$ and for the sinc quadrature spacing $k=0.6$. We observe that the error decays as the eigenvalue increases. Therefore, $\max_{\widetilde\lambda\ge 2}|e(\widetilde\lambda)| = |e(2)|$ and the sinc quadrature error is independent of the largest eigenvalue\footnote{We note that the monotonicity of the error $|e(\tilde\lambda)|$ depends on the choice of $\mathtt M$ and $\mathtt N$ with respect to $k$. Here our simulation is based on the balanced scheme \eqref{e:choose}.}.

\begin{figure}[htbp]
    \centering
    \begin{tabular}{c}
        \includegraphics[scale=0.2]{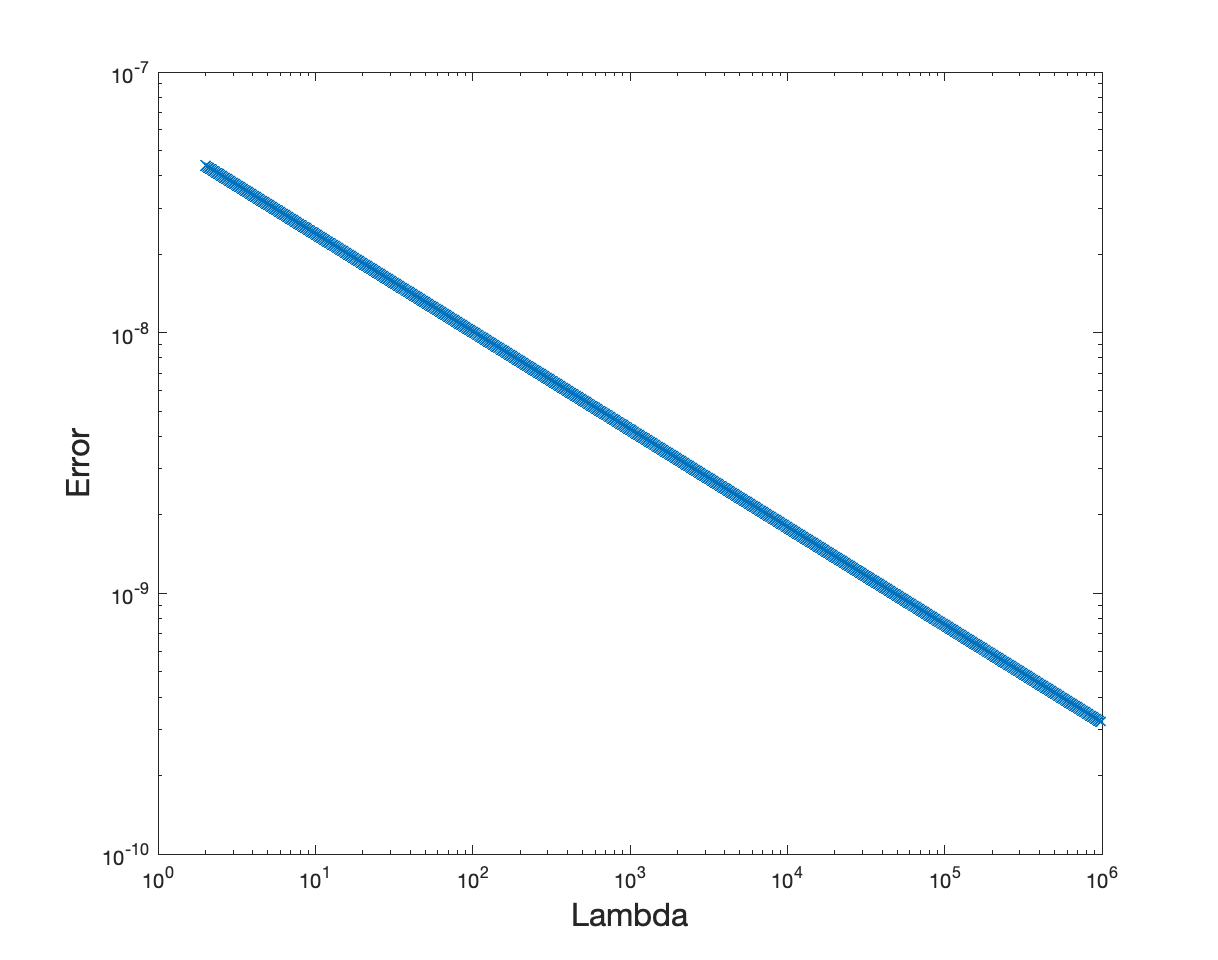}
    \end{tabular}
    \caption{Log-log plot of sinc approximation error $|e(\widetilde\lambda)|$ for $\widetilde\lambda\in [2,10^7]$.}
    \label{fig:sinc}
\end{figure}
\end{remark}

\section{Mean Square Norm Error Estimate} \label{s:weak}
In this section, we provide an error bound between the mean square norm of the exact solution $\widetilde u$ and that of its finite element approximation $U_k$. The main result is the following theorem.

\begin{theorem}[convergence in the mean square norm]\label{t:weak-convergence}
     Assume that $\tfrac{n-1}4 < s < 1$. Let $\widetilde u$ be the solution to \eqref{e:spde} and let $U_k$ be the approximation defined by \eqref{e:sol_ukh}. Then there holds\textcolor{black}{
     \[
          \bigg|\|\widetilde u\|_{L^2(\Omega;L^2(\gamma))}^2
          - \|U_k\|_{L^2(\Omega;L^2(\Gamma))}^2 \bigg|
          \lesssim \log(h^{-1})^3 h^{4s-(n-1)}+ C(h)h^2 + e^{-\pi^2/k},
     \]}
     where $C(h)$ is the constant that appears in Lemma~\ref{l:frac-geo-error}.
\end{theorem}

\begin{proof}
     Using the triangle inequality, we spilt the target error with the following five error terms
     \[
          \bigg|\|\widetilde u\|_{L^2(\Omega;L^2(\gamma))}^2
          - \|U_k\|_{L^2(\Omega;L^2(\Gamma))}^2 \bigg|
          \le \sum_{i=1}^5 E_i,
     \]
     where, thanks to \eqref{e:mean_square_UkPsi}, we have
     \[
          \begin{aligned}
               E_1 & := \bigg| \|\widetilde u\|_{L^2(\Omega; L^2(\gamma))}^2 - \|\widetilde u_k\|_{L^2(\Omega;L^2(\gamma))}^2 \bigg|
               = \bigg| \sum_{j=1}^\infty \big( \widetilde\lambda_{j}^{-2s} - \mathcal Q_k^{-s}(\widetilde \lambda_j)^2 \big)\bigg|,                                                                                                 \\
               E_2 & := \sum_{j=N+1}^\infty \mathcal Q_k^{-s}(\widetilde\lambda_j)^{2},                                                                                                                                              \\
               E_3 & := \bigg| \sum_{j=1}^N \big( \mathcal Q_k^{-s}(\widetilde\lambda_{j})^{2} - \mathcal Q_k^{-s}(\widetilde\Lambda_j^\gamma)^2 \big)\bigg|,                                                                        \\
               E_4 & := \bigg|\|\mathcal Q_k^{-s}(\widetilde L_{\mathcal T}) \widetilde W^\Psi\|_{L^2(\Omega;L^2(\gamma))}^2 - \|\mathcal Q_k^{-s}(L_{\mathcal T}) (\sigma P\widetilde W^\Psi)\|_{L^2(\Omega;L^2(\Gamma))}^2 \bigg|, \\
               E_5 & :=\bigg|\|\mathcal Q_k^{-s}(L_{\mathcal T}) (\sigma P\widetilde W^\Psi)\|_{L^2(\Omega;L^2(\Gamma))}^2 - \|U_k\|_{L^2(\Omega;L^2(\Gamma))}^2 \bigg|.                                                             
          \end{aligned}
     \]

     Note that \eqref{e:weak-norm-equal} guarantees that $E_5 = 0$ so we only need to estimate each error $E_i$, $i=1,\ldots,4$.

     \boxed{1}  For $E_1$, the sinc quadrature error formula \eqref{i:sclar-exp} together with \eqref{e:choose} directly implies that for $r=\tfrac{n-1}4+s$ we have
     \[
          |\widetilde\lambda_j^{-s} - \mathcal Q_k^{-s}(\widetilde\lambda_j)| \lesssim |e(\widetilde \lambda_j)|\widetilde\lambda_j^{-r/2} \lesssim e^{-\pi^2/k} \widetilde\lambda_j^{-r/2}
     \]
     so that
     $$
          |\widetilde\lambda_j^{-2s} - \mathcal Q_k^{-s}(\widetilde\lambda_j)^2| \lesssim e^{-\pi^2/k}\widetilde\lambda_j^{-r/2}( \widetilde \lambda_j^{-s} + \mathcal{Q}_k^{-s}(\widetilde\lambda_j)) \textcolor{black}{\lesssim e^{-\pi^2/k}\widetilde\lambda_j^{-r},}
     $$
     where we used
     \begin{equation}\label{e:stab_scala_sinc}
          \textcolor{black}{\mathcal{Q}_k^{-s}(\widetilde \lambda_j) \lesssim \frac{1}{s-r/2}\widetilde \lambda_j^{-r/2}}
     \end{equation}
     with $r=\tfrac{n-1}{4}+s$ to justify the last inequality.
     Consequently, we find that
     \[
          E_1   \lesssim  e^{-\pi^2/k}\sum_{j=1}^\infty\textcolor{black}{ \widetilde\lambda_j^{-r} } \lesssim e^{-\pi^2/k} \| \widetilde w \|_{L^2(\Omega;H^{-r}(\gamma))}^2\lesssim e^{-\pi^2/k}.
     \]

     \boxed{2} We invoke again the stability of the sinc quadrature \eqref{e:stab_scala_sinc} with $r=2s-\epsilon$ with $\epsilon = (2s-(n-1)/2)/\ln(h^{-1})$ together with Weyl's law \eqref{i:weyl} to estimate $E_2$\textcolor{black}{
     \[
     \begin{aligned}
          E_2 \lesssim \epsilon^{-2}\sum_{j=N+1}^\infty \widetilde\lambda_j^{-(2s-\epsilon)}
         & \lesssim \epsilon^{-2}\sum_{j=N+1}^\infty j^{-2(2s-\epsilon)/(n-1)} \\
         & \lesssim \epsilon^{-3}N^{1-(4s-2\epsilon)/(n-1)}
          \lesssim (\log(h^{-1}))^{3} h^{4s-(n-1)} .
    \end{aligned}
     \]}

     \boxed{3} For $E_3$, the stability of the sinc quadrature \eqref{e:stab_scala_sinc} along with the error estimates on the approximation of the eigenvalues (Lemma~\ref{l:eigen-error}) yield
     \[
          \begin{aligned}
               \mathcal Q_k^{-s}(\widetilde\lambda_j) - \mathcal Q_k^{-s}(\widetilde\Lambda_j^\gamma)
                & = \frac{k\sin(\pi s)}{\pi}\sum_{l = -\mathtt N}^{\mathtt M} e^{(1-s)y_l}
               \bigg(\frac{1}{e^{y_l} +\widetilde\lambda_j} - \frac{1}{e^{y_l} +\widetilde\Lambda_j^\gamma}\bigg)                              \\
                & =  \frac{k\sin(\pi s)}{\pi}\sum_{l = -\mathtt N}^{\mathtt M} e^{(1-s)y_l}
               \frac{\widetilde\Lambda_j^\gamma-\widetilde\lambda_j}{(e^{y_l} +\widetilde\lambda_j)(e^{y_l} +\widetilde\Lambda_j^\gamma)}      \\
                & \le \frac{\widetilde\Lambda_j^\gamma - \widetilde\lambda_j}{\widetilde\Lambda_j^\gamma}\bigg(\frac{k\sin(\pi s)}{\pi}\sum_{l=-\mathtt N}^{\mathtt M}
               e^{(1-s)y_l}  \frac{1}{e^{y_l} +\widetilde\lambda_j} \bigg)                                                                             \\
                & = \frac{\widetilde\Lambda_j^\gamma - \widetilde\lambda_j}{\widetilde\Lambda_j^\gamma} \mathcal Q_k^{-s}(\widetilde\lambda_j)
             \textcolor{black}{  \lesssim \epsilon^{-1}\widetilde\lambda_j^{1-s+\epsilon/2}h^2.}
          \end{aligned}
     \]
     Here again $\epsilon = (2s-(n-1)/2)/\ln(h^{-1})$. Since $\kappa^2 \leq \widetilde\lambda_j\le \widetilde\Lambda_j^{\gamma}$ thanks to Lemma~\ref{l:eigen-error}, we deduce that\textcolor{black}{
     \[
          \begin{aligned}
               E_3 & \lesssim  \sum_{j=1}^N |\mathcal Q_k^{-s}(\widetilde\lambda_j) - \mathcal Q_k^{-s}(\widetilde\Lambda_j^{\gamma})| |\mathcal Q_k^{-s}(\widetilde\lambda_j)+ \mathcal Q_k^{-s}(\widetilde\Lambda_j^\gamma)| \\
                   & \lesssim  \sum_{j=1}^N |\mathcal Q_k^{-s}(\widetilde\lambda_j) - \mathcal Q_k^{-s}(\widetilde\Lambda_j^{\gamma}) | \epsilon^{-1}\widetilde\lambda_j^{-s+\epsilon/2}  \lesssim \epsilon^{-2}h^2\sum_{j=1}^N \widetilde\lambda_j^{1-2s+\epsilon}          \\
                   & \lesssim \epsilon^{-2}h^2 \sum_{j=1}^N j^{(2-4s+2\epsilon)/(n-1)}     \lesssim  \epsilon^{-2}h^2 N^{1+(2-4s+2\epsilon)/(n-1)} \lesssim (\log(h^{-1}))^2h^{4s-(n-1)}.
          \end{aligned}
     \]}

     \boxed{4} We now focus on $E_4$. Using the triangle inequality and Lemma~\ref{l:frac-geo-error} as well as Remark~\ref{r:sol-bounded}, we have

     \begin{equation*}
          \begin{split}
               E_4 & \le \bigg |\|\mathcal Q_k^{-s}(\widetilde L_{\mathcal T}) \widetilde W^\Psi\|_{L^2(\Omega;L^2(\gamma))}^2 - \|P\mathcal Q_k^{-s}(\widetilde L_{\mathcal T}) \widetilde W^\Psi\|_{L^2(\Omega;L^2(\Gamma))}^2\bigg| \\
               & \quad\quad+ \bigg|\|P\mathcal Q_k^{-s}(\widetilde L_{\mathcal T}) \widetilde W^\Psi\|_{L^2(\Omega;L^2(\Gamma))}^2 - \|\mathcal Q_k^{-s}(L_{\mathcal T}) (\sigma P\widetilde W^\Psi)\|_{L^2(\Omega;L^2(\Gamma))}^2 \bigg|   \\
               & \lesssim \|1-\sigma\|_{L^\infty(\Gamma)} \|\mathcal Q_k^{-s}(\widetilde L_{\mathcal T}) \widetilde W^\Psi\|_{L^2(\Omega;L^2(\gamma))}^2 + C(h) h^2,
          \end{split}
     \end{equation*}
     where $C(h)$ is as in Lemma~\ref{l:frac-geo-error} and where for the second term on the right hand side we used \eqref{i:discrete-solution-bound-A} and \eqref{i:discrete-solution-bound-B}.
     The geometric consistency \eqref{i:sigma} together with the stability estimate \eqref{i:discrete-solution-bound-B} imply
     $$
          E_4 \lesssim h^2
          +  C(h) h^2 \lesssim C(h) h^2.
     $$

     The proof is complete by combing the estimates for $E_i$ with $i=1,\ldots,5$.
\end{proof}
\begin{remark}\label{r:rate-limited}
     We note that for two dimensional surfaces $(n=2)$, the rate of convergence is limited to $O(h^2)$ even when $s>\tfrac34$.
     This $O(h^2)$ limitation is due to the geometric approximation of the area $\|1-\sigma\|_{L^\infty(\Gamma)}$ to derive the estimate for $E_4$ in the proof of Theorem~\ref{t:weak-convergence}.
     We anticipate that the optimal order of convergence $O(h^{4s-(n-1)})$ can be obtained upon using higher order surface approximations. In fact, when $\gamma$ is of class $C^{p+1}$ and polynomial of degree $p$ are used to define the approximation $\Gamma$ of $\gamma$, one has $\| 1-\sigma \|_{L^\infty(\Gamma)}\lesssim h^{p+1}$ (see e.g. \cite{demlow09}).
\end{remark}

\section{Numerical Illustration}\label{s:numeric}

In this section, we perform several numerical experiments to illustrate the proposed numerical method. In particular, we compute the strong and mean square norm errors as well as illustrate the effect of the parameters $s$ and $\kappa$ by plotting one realization of the approximated random field $U_k$ and by evaluating the covariance ${\rm cov}_{U_k}(\bx,\bx')$ at some points $\bx,\bx'\in\Gamma$. Recall that $U_k$ is a linear combination of the $\{U^l\}_{l=-{\mathtt M}}^{\mathtt N}$, see \eqref{def:LS_random_l}, and that a realization $U^l(\cdot,\omega)\in\mathbb{V}(\mathcal T)$ of $U^l$ is obtained by solving the linear system \eqref{def:LS_random_l} for a realization ${\bf z}(\omega)$ of ${\bf z}\sim\mathcal{N}(\mathbf 0,I_{N\times N})$.

For the numerical error analysis, we focus on the error due to the finite element discretization. Therefore, from now on, we set $k=0.6$ for the quadrature spacing and choose $\mathtt N$ and $\mathtt M$ according to \eqref{e:choose}, thus yielding (up to a multiplicative constant) a sinc quadrature error of order $e^{-\pi^2/k}\approx 7.1803\cdot 10^{-8}$. The numerical implementation is based on the \texttt{deal.ii} library \cite{dealII94} and the visualization is done with \texttt{ParaView} \cite{ayachit2015}.

\subsection{Gaussian Mat\'ern Random Field on the Unit Sphere}

In this first example, we let $\gamma = \mathbb S^2$ be the sphere in $\mathbb{R}^3$ parametrized using the spherical coordinates $(\theta,\varphi)\in[0,\pi]\times[0,2\pi)$, where $\theta$ and $\varphi$ are the elevation angle (latitude) and azimuth angle (longitude), respectively. Using the spherical harmonic functions, see for instance \cite{MP2011,LS2015,jansson2021surface}, the formal KL expansion \eqref{e:white_noise} of the white noise can be written
\begin{equation} \label{eqn:KL_Sphere_WN}
     \begin{aligned}
          \widetilde \omega(\bx,\omega) & = \Big\langle\sum_{l=0}^{\infty}\Big[\xi_{l,0}^1(\omega) q_{l,0}(\theta) \\ &\qquad+\sqrt{2}\sum_{m=1}^{l} q_{l,m}(\theta)\left(\xi_{l,m}^1(\omega)\cos(m\varphi)+\xi_{l,m}^2(\omega)\sin(m\varphi)\right)\Big],\cdot\Big\rangle
     \end{aligned}
\end{equation}
and the solution to the SPDE \eqref{e:spde} reads
\begin{eqnarray} \label{eqn:KL_Sphere_sol}
     \widetilde u(\bx,\omega) & = & \sum_{l=0}^{\infty}(\kappa^2+l(l+1))^{-s}\Big[\xi_{l,0}^1(\omega) q_{l,0}(\theta) \nonumber \\
          & & +\sqrt{2}\sum_{m=1}^{l} q_{l,m}(\theta)\left(\xi_{l,m}^1(\omega)\cos(m\varphi)+\xi_{l,m}^2(\omega)\sin(m\varphi)\right)\Big],
\end{eqnarray}
where the $\kappa^2+l(l+1)$ are the eigenvalues of the operator $L$ (multiplicity $2l+1$). Here $$\xi_{l,m}^i\stackrel{i.i.d.}{\sim} \mathcal N(0,1), \quad l\in\mathbb{N}_0:=\mathbb{N} \cup \{ 0 \}, \,\, m=0,\ldots,l, \,\, i=1,2,$$
and $\bx=(\sin(\theta)\cos(\varphi),\sin(\theta)\sin(\varphi),\cos(\theta))\in\gamma$. Moreover, for $l\in\mathbb{N}_0$ and $m=0,\ldots,l$,
$$q_{l,m}(\theta):=\sqrt{\frac{2l+1}{4\pi}\frac{(l-m)!}{(l+m)!}}P_{l,m}(\cos(\theta))$$
with $P_{l,m}(\mu)$ the Legendre functions
$$P_{l,m}(\mu):=(-1)^m(1-\mu^2)^{\frac{m}{2}}\frac{\partial^m}{\partial\mu^m}P_{l}(\mu), \quad \mu\in[-1,1],$$
associated to the Legendre polynomials
$$P_{l}(\mu):=\frac{1}{2^{l}l!}\frac{\partial^{l}}{\partial\mu^{l}}(\mu^2-1)^{l}, \quad \mu\in[-1,1].$$
Note that the (complex-valued) spherical harmonic functions are given in spherical coordinates by $y_{l,m}(\theta,\varphi):=q_{l,m}(\theta)e^{im\varphi}$ for $l\in\mathbb{N}_0$ and $m=0,\ldots,l$ and $y_{l,m}:=(-1)^m\bar y_{l,m}$ for $l\in\mathbb{N}$ and $m=-l,\ldots,-1$. Then for $l\in\mathbb{N}_0$ and $m=-l,\ldots,l$, $y_{l,m}$ is an eigenfunction of the Laplace--Beltrami operator associated to the eigenvalue $l(l+1)$. Finally, we mention that $\widetilde u$ in \eqref{eqn:KL_Sphere_sol} satisfies
\begin{equation} \label{e:norm_u}
     \|\widetilde u\|_{L^2(\Omega;L^2(\gamma))}^2=\sum_{l=0}^{\infty}(\kappa^2+l(l+1))^{-2s}(2l+1).
\end{equation}


\subsubsection{Strong and Mean Square Norm Errors}
We first compute the strong and mean square norm errors given, respectively, by
$$\|P\widetilde u-U_k\|_{L^2(\Omega;L^2(\Gamma))} \sim  \|\widetilde u - P^{-1}U_k\|_{L^2(\Omega;L^2(\gamma))}
$$
and
$$
     \left|\|\tilde u\|_{L^2(\Omega;L^2(\gamma))}^2-\|U_k\|_{L^2(\Omega;L^2(\Gamma))}^2\right|,
$$
where $U_k$ is the approximation given in \eqref{e:sol_ukh}. However, these quantities cannot be computed exactly in practice, and some approximations are required. First, we accurately compute the norm of $\widetilde u$ (still denoted $\|\widetilde u\|_{L^2(\Omega;L^2(\gamma))}^2$ below) using \eqref{e:norm_u} but keeping the first $100000$ terms. For the other terms, we use the vanilla Monte Carlo method with sample size $K$ to compute the expected values. Now for the strong error, since $\tilde u$ involves an infinite sum, we replace it by $\widetilde u_{\mathtt L}=L^{-s}\widetilde w_{\mathtt L}$, where $\widetilde u_{\mathtt L}$ and $\widetilde w_{\mathtt L}$ are defined as in \eqref{eqn:KL_Sphere_WN} and \eqref{eqn:KL_Sphere_sol}, respectively, but with the index $l$ running from $0$ to some positive integer $\mathtt L$. Then we have
\begin{equation} \label{e:norm_uL}
     \|\widetilde u_{\mathtt L}\|_{L^2(\Omega;L^2(\gamma))}^2=\sum_{l=0}^{\mathtt L}(\kappa^2+l(l+1))^{-2s}(2l+1)
\end{equation}
and
\begin{equation}
     \|\widetilde w_{\mathtt L}\|_{L^2(\Omega;L^2(\gamma))}^2=\sum_{l=0}^{\mathtt L}(2l+1)=(\mathtt L +1)^2,
\end{equation}
and we refer to \cite{LS2015} for an analysis of the truncation error $\|\tilde u-\tilde u_{\mathtt L}\|_{L^2(\Omega;L^2(\gamma))}$.
Moreover, we need to have comparable samples for the (truncated) exact solution $\widetilde u_{\mathtt L}$ and the approximation $U_k$. This would require the computation of ${\boldsymbol \alpha}$ with entries given in \eqref{def:comp_rhs}, and thus the computation of the $L^2(\gamma)$ projection of the eigenfunctions onto the conforming finite element space $\widetilde{\mathbb V}(\mathcal T)$. In particular, we cannot use the strategy described in Section~\ref{s:alpha}, namely replace $\boldsymbol{\alpha}$ by $G\bsz$ with $\bsz\sim\mathcal{N}(0,I_{N\times N})$ and $GG^T=M$. As an alternative, we follow \cite{bolin2020numerical} and replace $U_k$ by $U_{\mathtt L,k}$ obtained by using the data vector with entries $(\sigma P\tilde w_{\mathtt L},\Phi_i)_{\Gamma}$, $i=1,\ldots,N$. To sum up, we report below the errors
$$e_{\rm strong}:= \left(\frac{1}{K}\sum_{i=1}^{K}\|P\tilde u_{\mathtt L}(\cdot,\omega_i)- U_{\mathtt L,k}(\cdot,\omega_i)\|_{L^2(\Gamma)}^2\right)^{\frac 1 2}$$
and
$$e_{\rm weak}:= \left|\|\tilde u\|_{L^2(\Omega;L^2(\gamma))}^2-\frac{1}{K}\sum_{i=1}^{K}\|U_k(\cdot,\omega_i)\|_{L^2(\Gamma)}^2\right|,$$
as well as $e_{\sigma}:=\|1-\sigma\|_{L^{\infty}(\Gamma)}$, for different meshes and different values of the parameters $s$ and $\kappa$. Recall that according to Theorems~\ref{t:strong} and~\ref{t:weak-convergence}, the theoretical convergence rate is $2s-1$ for the strong error (up to a logarithm term) and $4s-2$ for the mean square error. In Table~\ref{tab:error_strong} we report the strong error $e_{\rm strong}$ for $\mathtt L=100$ and $K=10000$. In all cases we observe approximately the convergence rate $O(h^2)$ predicted by Theorem 4.2 in \cite{bonito2021approximation} for smooth right-hand sides (in physical space), see also Proposition 4.2 in \cite{jansson2021surface}.
In passing, we mention that numerical instabilities may arise when $\mathtt{L}$ is large because of large values assumed by the Legendre functions $P_{l,m}$ (for instance $P_{100,98}(0)=-1.675\cdot 10^{184}$ and $P_{100,100}(0)=6.666\cdot 10^{186}$). However, we do not clearly observe such instability, see   Figure~\ref{fig:sphere_errors_strong}.



\begin{table}[htbp]
     \centering
     \begin{tabular}{|c|c|c|r|c|r|r|c|c|}
          \cline{4-9}
          \multicolumn{3}{c|}{ } & \multicolumn{3}{c|}{$s=0.75$} & \multicolumn{3}{c|}{$s=0.9$} \\
          \hline
          $N$ & $h$ & $e_{\sigma}$ & $\kappa=0.5$ & $\kappa=2$ & $\kappa=8$ & $\kappa=0.5$ & $\kappa=2$ & $\kappa=8$ \\
          \hline
8 & 1.633 & 2.000 & 65.488 & 11.199 & 1.913 & 79.794 & 8.672 & 1.024 \\
26 & 1.000 & 0.3032 & 23.258 & 5.418 & 1.294 & 27.780 & 4.007 & 0.677 \\
98 & 0.541 & 0.0778 & 11.367 & 3.078 & 0.976 & 13.404 & 2.191 & 0.489 \\
386 & 0.276 & 0.0194 & 4.550 & 1.334 & 0.532 & 5.328 & 0.916 & 0.247 \\
1538 & 0.139 & 0.0048 & 0.244 & 0.202 & 0.192 & 0.180 & 0.070 & 0.061 \\
6146 & 0.070 & 0.0012 & 0.097 & 0.097 & 0.095 & 0.029 & 0.028 & 0.028 \\
          \hline
     \end{tabular}
     \caption{Error $e_{\rm strong}$ using $K=10000$ Monte Carlo samples and $\mathtt{L}=100$ for the truncation of the white noise.}
     \label{tab:error_strong}
\end{table}

\begin{figure}[htbp]
     \centering
     \includegraphics[width=0.55\textwidth]{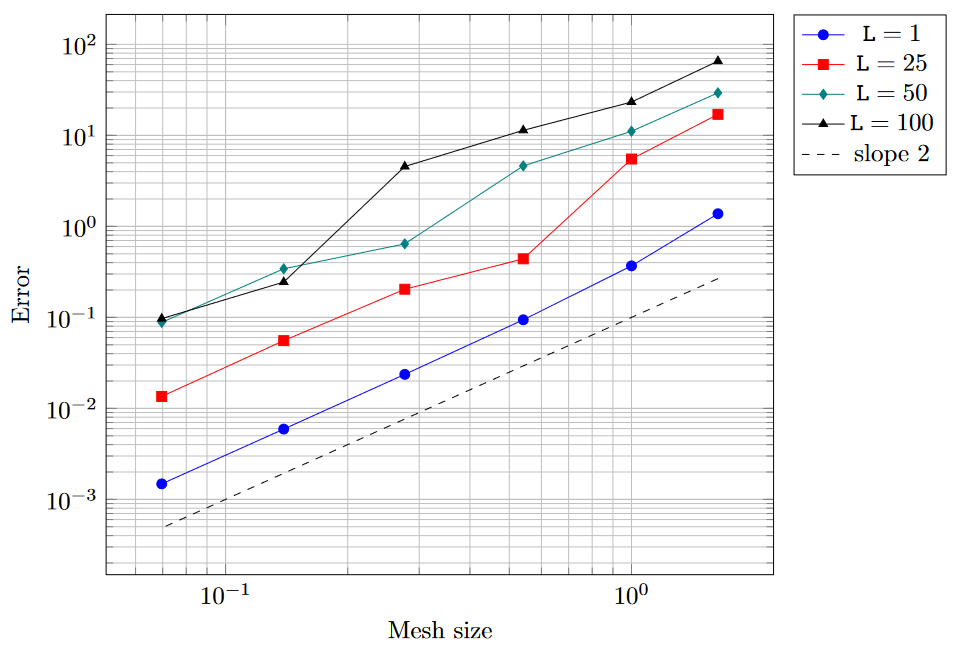}
     \caption{Errors $e_{\rm strong}$ for $\kappa=0.5$ and different values of the truncation parameter $\mathtt{L}$ using $K=10000$ Monte Carlo samples.}
     \label{fig:sphere_errors_strong}
\end{figure}

Regarding the mean square error $e_{\rm weak}$, we report its value along with $\|U_k\|_{L^2(\Omega;L^2(\Gamma))}^2$ for different fractional powers $s$ and the cases $\kappa=2$ and $\kappa=8$ in Tables~\ref{tab:weak_error_kappa2} and \ref{tab:weak_error_kappa8}, respectively.
These tables are supplemented with Figure~\ref{fig:sphere_errors_weak} to better comprehend the evolution of the mean square error as the meshsize varies. We observe that the convergence rates observed numerically match the predictions from Theorem~\ref{t:weak-convergence} when the finite element error is dominant, namely when $s$ is small (i.e. the solution is not smooth) or when $\kappa$ is large (i.e. the correlation length is small). In the other cases,  other sources of error like the Monte Carlo  and the sinc quadrature errors affect the convergence with respect to $h$.



\begin{table}[htbp]
\centering
\begin{tabular}{|c|c|c|c|c|c|c|c|}
          \cline{3-8}
          \multicolumn{2}{c|}{ } &  \multicolumn{2}{c|}{$s=0.625$} & \multicolumn{2}{c|}{$s=0.75$} & \multicolumn{2}{c|}{$s=0.9$} \\
          \hline
          $N$ & $h$ & $\|U_k\|^2$ & $e_{\rm weak}$ & $\|U_k\|^2$ & $e_{\rm weak}$ & $\|U_k\|^2$ & $e_{\rm weak}$ \\
          \hline
98 & 0.542 & 1.4399 & 1.4391 & 0.7554 & 0.2899 & 0.3738 & 0.0690 \\
386 & 0.276 & 1.8060 & 1.0729 & 0.8751 & 0.1701 & 0.4103 & 0.03247 \\
1538 & 0.139 & 2.0978 & 0.7812 & 0.9461 & 0.0992 & 0.4248 & 0.0180 \\
6146 & 0.070 & 2.3210 & 0.5579 & 0.9903 & 0.0550 & 0.4336 & 0.0092 \\
24578 & 0.035 & 2.4761 & 0.4028 & 1.0087 & 0.0366 & 0.4339 & 0.0089 \\
\hline
     \end{tabular}
     \caption{Mean square error $e_{\rm weak}$ using $K=1000$ when $\kappa=2$, in which case $\|\tilde u\|_{L^2(\Omega;L^2(\gamma))}^2$ is equal to $2.87891$ for $s=0.625$, $1.04528$ for $s=0.75$, and $0.44277$ for $s=0.9$.}
     \label{tab:weak_error_kappa2}
\end{table}

\begin{table}[htbp]
\centering
\begin{tabular}{|c|c|c|c|c|c|c|c|}
          \cline{3-8}
          \multicolumn{2}{c|}{ } &  \multicolumn{2}{c|}{$s=0.625$} & \multicolumn{2}{c|}{$s=0.75$} & \multicolumn{2}{c|}{$s=0.9$} \\
          \hline
          $N$ & $h$ & $\|U_k\|^2$ & $e_{\rm weak}$ & $\|U_k\|^2$ & $e_{\rm weak}$ & $\|U_k\|^2$ & $e_{\rm weak}$ \\
          \hline
98 & 0.542 & 0.2605 & 1.1429 & 0.0813 & 0.1694 & 0.0203 & 0.0248 \\
386 & 0.276 & 0.4684 & 0.9351 & 0.1329 & 0.1177 & 0.0303 & 0.0148 \\
1538 & 0.139 & 0.6859 & 0.7175 & 0.1774 & 0.0732 & 0.0375 & 0.0076 \\
6146 & 0.070 & 0.8741 & 0.5293 & 0.2083 & 0.0423 & 0.0415 & 0.0036 \\
24578 & 0.035 & 1.0250 & 0.3784 & 0.2278 & 0.0229 & 0.0435 & 0.0015 \\
\hline
     \end{tabular}
     \caption{Mean square error $e_{\rm weak}$ using $K=1000$ when $\kappa=8$, in which case $\|\tilde u\|_{L^2(\Omega;L^2(\gamma))}^2$ is equal to $1.40341$ for $s=0.625$, $0.25063$ for $s=0.75$, and $0.04506$ for $s=0.9$.}
     \label{tab:weak_error_kappa8}
\end{table}

\begin{figure}[htbp]
     \centering
     \includegraphics[width=0.44\textwidth]{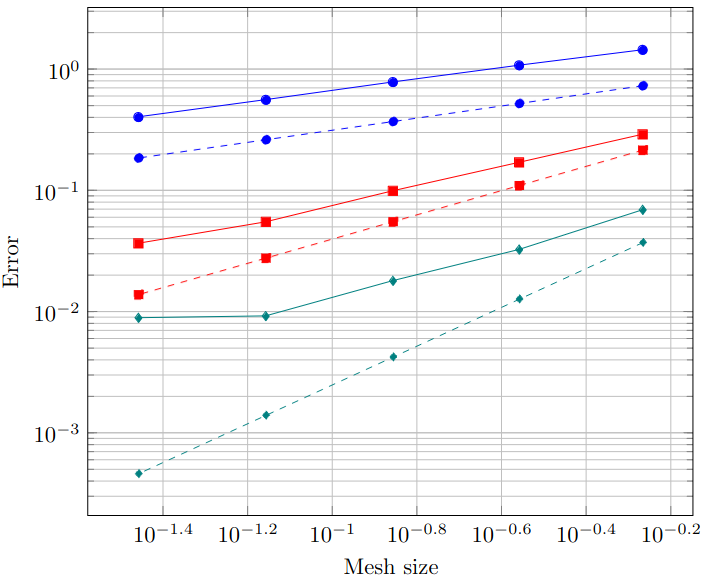}
     \includegraphics[width=0.55\textwidth]{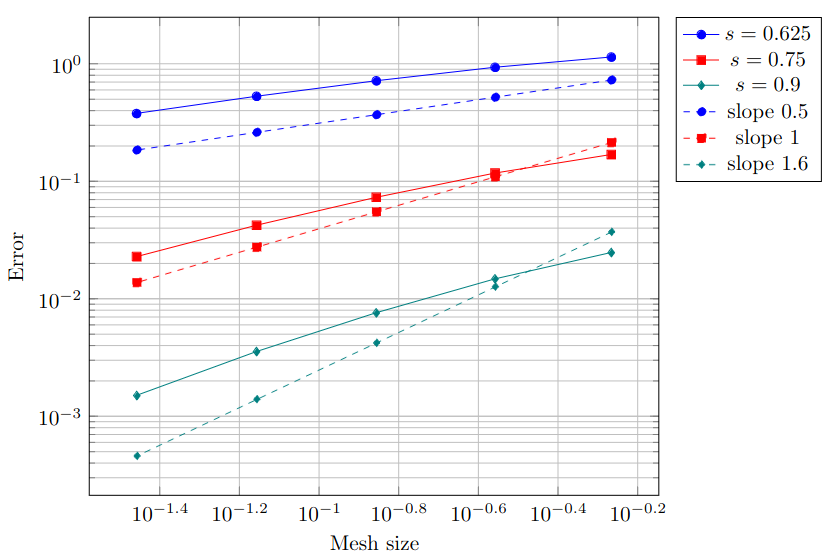}
     \caption{Errors $e_{\rm weak}$ for $\kappa=2$ (left) and $\kappa=8$ (right) using $K=1000$ Monte Carlo samples. The dashed lines indicate the behavior predicted by Theorem~\ref{t:weak-convergence}.}
     \label{fig:sphere_errors_weak}
\end{figure}

\subsubsection{Effect of the Parameters $s$ and $\kappa$}
The fractional power $s$ determines the regularity of the random field. To illustrate this, we give on Figure \ref{fig:one_real_sphere} one realization of $U_k$ for $s=0.55,0.75,0.95$ using $N=1538$ DoFs ($h=0.1391$). For the color map, blue indicates negative values while red indicates positive values, the range of values being $[-2.2614, 1.7720]$ for $s=0.55$, $[-1.1898,0.6849]$ for $s=0.75$, and $[-0.8278,0.2587]$ for $s=0.95$.



\begin{figure}[htbp]
     \centering
     \includegraphics[width=0.3\textwidth]{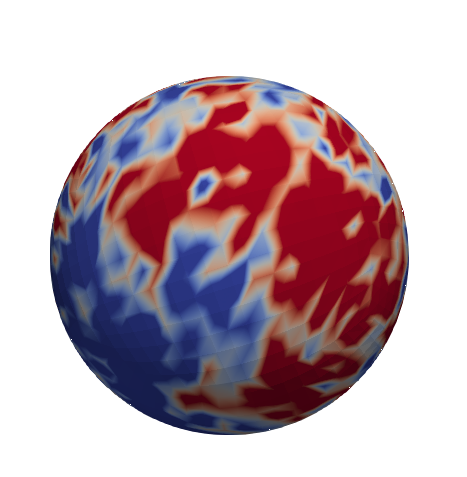}
     \includegraphics[width=0.3\textwidth]{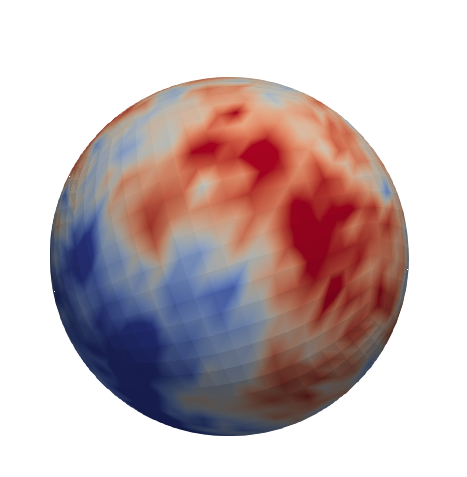}
     \includegraphics[width=0.3\textwidth]{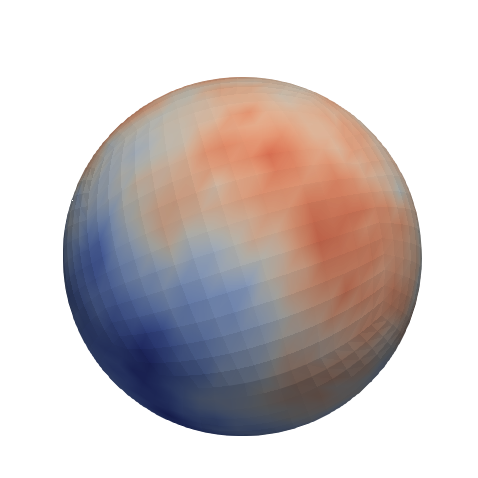}
     \caption{Numerical solution $U_k(\cdot,\omega_m)$ to Problem \eqref{e:spde} on the unit sphere when $\kappa=0.5$ and $s=0.55$ (left), $s=0.75$ (middle), and $=0.95$ (right).}
     \label{fig:one_real_sphere}
\end{figure}

To see the influence of the parameter $\kappa$, which is inversely proportional to the correlation length, we compute the approximated covariance at different points, namely
\begin{equation} \label{def:covU}
     {\rm cov}_{U_k}(\bx,\bx'):=\frac{1}{K-1}\sum_{i=1}^{K}\left(U_k(\bx,\omega_i)-\bar U_k(\bx)\right)\left(U_k(\bx',\omega_i)-\bar U_k(\bx')\right),
\end{equation}
where
\begin{equation*}
     \bar U_k(\bx) := \frac{1}{K}\sum_{i=1}^M U_k(\bx, \omega_i).
\end{equation*}
We take $K=10 000$ for both the approximation of the mean field $\bar U_k$ and the covariance function ${\rm cov}_{U_k}$, but using two different sets of samples, and we take $N=1538$ DoFs ($h=0.139$). The results are reported in Table \ref{tab:SphereCov}. We observe numerically that the smaller $\kappa$ the larger the covariance, and that the evaluation of ${\rm cov}_{U_k}$ at pairs of points with the same (geodesic) distance yields comparable values.

\begin{table}[htbp]
     \centering
     \begin{tabular}{|c|c|c||c|c|}
          \cline{2-5}
          \multicolumn{1}{c|}{ } & \multicolumn{2}{|c||}{$s=0.75$} & \multicolumn{2}{|c|}{$s=0.9$}                             \\
          \cline{2-5}
          \multicolumn{1}{c|}{ } & $\kappa=0.5$                    & $\kappa=2$                    & $\kappa=0.5$ & $\kappa=2$ \\
          \hline
          $(\bx_1,\bx_2)$        & 0.623685                        & 0.005944                      & 0.951398     & 0.004374   \\
          \hline
          $(\bx_1,\bx_3)$        & 0.577621                        & 0.001588                      & 0.909999     & 0.000980   \\
          \hline
          $(\bx_2,\bx_3)$        & 0.617366                        & 0.004903                      & 0.945554     & 0.003722   \\
          \hline
     \end{tabular}
     \caption{Covariance function defined in \eqref{def:covU} evaluated at the following points: the \emph{South Pole} $\bx_1=(0,0,-1)$, the point $\bx_2=(0,1,0)$ which lies on the Equator, and the \emph{North Pole} $\bx_3=(0,0,1)$.}
     \label{tab:SphereCov}
\end{table}

\subsection{Gaussian Mat\'ern Random Field on a Torus}

Here the surface $\gamma$ is a torus with parametrization
\begin{equation}\label{def:torus}
     (x,y,z) = ((R+r\cos(\theta))\cos(\phi),r\sin(\theta),(R+r\cos(\theta))\sin(\phi))
\end{equation}
for $\theta,\phi\in[0,2\pi)$. In what follows, we set $r=0.5$ and $R=2$. One realization for the case $\kappa=0.5$ is given on Figure \ref{fig:one_real_torus} for $s=0.55,0.75,0.95$ using $N=5120$ ($h=0.1386$). As above, blue indicates negative values while red stands for positive values, the range of values being in this case $[-2.6314, 3.0533]$ for $s=0.55$, $[-1.5083,1.5491]$ for $s=0.75$, and $[-1.2159,0.7646]$ for $s=0.95$.


\begin{figure}[htbp]
     \centering
     \includegraphics[width=0.3\textwidth]{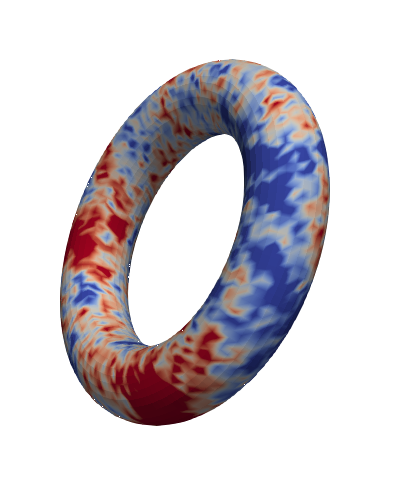}
     \includegraphics[width=0.3\textwidth]{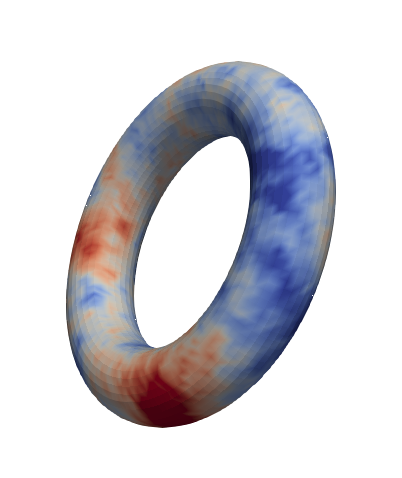}
     \includegraphics[width=0.3\textwidth]{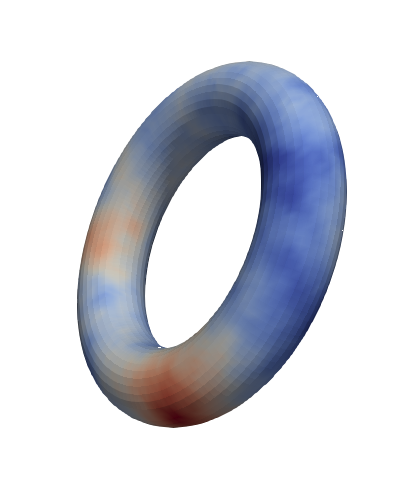}
     \caption{Numerical solution to Problem \eqref{e:spde} on the torus defined in \eqref{def:torus} when $\kappa=0.5$ and $s=0.55$ (left), $s=0.75$ (middle), and $s=0.95$ (right).}
     \label{fig:one_real_torus}
\end{figure}

As for the sphere, we evaluate the covariance function ${\rm cov}_{U_k}$ at different points and for different values of the parameters $s$ and $\kappa$, see Table \ref{tab:SphereCov} for the results obtained when $N=1280$ ($h=0.2757$). As above, we observe that the covariance is inversely proportional to $\kappa$ and the inverse of the geodesic distance between the two points.

\begin{table}[htbp]
     \centering
     \begin{tabular}{|c|c|c||c|c|}
          \cline{2-5}
          \multicolumn{1}{c|}{ } & \multicolumn{2}{|c||}{$s=0.75$} & \multicolumn{2}{|c|}{$s=0.9$}                             \\
          \cline{2-5}
          \multicolumn{1}{c|}{ } & $\kappa=0.5$                    & $\kappa=2$                    & $\kappa=0.5$ & $\kappa=2$ \\
          \hline
          $(\bx_1,\bx_2)$        & 0.377470                        & 0.015192                      & 0.505575     & 0.010112   \\
          \hline
          $(\bx_1,\bx_3)$        & 0.360484                        & 0.006877                      & 0.497597     & 0.005097   \\
          \hline
          $(\bx_2,\bx_3)$        & 0.401743                        & 0.017716                      & 0.529588     & 0.011722   \\
          \hline
     \end{tabular}
     \caption{Covariance function defined in \eqref{def:covU} evaluated at the following points: $\bx_1=(1.5,0,0)$, $\bx_2=(2,0.5,0)$, and $\bx_3=(2.5,0,0)$ which correspond to the cases $\phi=0$ and $\theta=\pi,0,\pi/2$, respectively.}
     \label{tab:TorusCov}
\end{table}


\appendix

\section{Proof of Lemma~\ref{l:frac-geo-error}} \label{sec:appendixB}
We follow the argument from Section~4 of \cite{bonito2021approximation}.
Note that from the definition of $\mathcal Q_k^{-s}$, we have
$$
     P \mathcal Q_k^{-s}(\widetilde L_{\mathcal T}) \widetilde W^\Psi - \mathcal Q_k^{-s}(L_{\mathcal T}) (\sigma P\widetilde W^\Psi)
     = \frac{k \sin(\pi s)}{\pi} \sum_{l=-\mathtt M}^{\mathtt N} e^{(1-s) y_l} \mathcal E_l \widetilde W^\Psi
$$
where  $\mathcal E_l: \widetilde{\mathbb V}(\mathcal T) \rightarrow \mathbb V(\mathcal T)$ is given by
\begin{equation}\label{e:diff}
     \mathcal E_l  := P(\mu_l I + \widetilde L_{\mathcal T})^{-1} - (\mu_l I + L_{\mathcal T})^{-1}(\sigma P)
\end{equation}
with $\mu_l := e^{y_l}$ and $y_l$ are the sinc quadrature points; see \eqref{e:approx_quad}.
The decomposition of the operator $\mathcal E_l$
\[
     \begin{aligned}
          \mathcal E_l & = (\mu_l I + L_{\mathcal T})^{-1}
               [(\mu_l I + L_{\mathcal T})P - \sigma P(\mu_l I + \widetilde L_{\mathcal  T})]
          (\mu_l I + \widetilde L_{\mathcal T})^{-1}                                                \\
                       & = L_{\mathcal T}(\mu_l I + L_{\mathcal T})^{-1}
          (P\widetilde T_{\mathcal T} - T_{\mathcal T}\sigma P)
          \widetilde L_{\mathcal T} (\mu_l I + \widetilde L_{\mathcal T})^{-1}                      \\
                       & \qquad
          \mu_l(\mu_l I + L_{\mathcal T})^{-1} (1-\sigma)P(\mu_l I +\widetilde L_{\mathcal T})^{-1} \\
                       & =: \mathcal E_l^1 + \mathcal E_l^2 .
     \end{aligned}
\]
is instrumental in the error analysis below.

The next result recall estimate for some terms in the above decomposition. We refer to Lemma~4.5 in \cite{bonito2021approximation} for more details.
\begin{lemma}\label{l:discrete-operator}
     Let $p\in [-1,1]$ and $q\in \mathbb R$ be such that $p+q\in[0,2]$. Then for any $\mu>0$ and any $\widetilde F\in \widetilde{\mathbb V}(\mathcal T)$ there holds
     \begin{equation}\label{i:appendix-esti-1}
          \|\widetilde L_{\mathcal T}(\mu I +\widetilde L_{\mathcal T} )^{-1}\widetilde F\|_{H^{-p}(\gamma)} \lesssim \mu^{-(p+q)/2} \|\widetilde L_{\mathcal T}^{q/2}\widetilde F\|_{L^2(\gamma)}.
     \end{equation}
     Moreover, for any $F\in \mathbb V(\mathcal T)$ we have
     \begin{equation}\label{i:appendix-esti-2}
          \|L_{\mathcal T}(\mu I + L_{\mathcal T} )^{-1} F\|_{L^2(\Gamma)} \lesssim \mu^{-1/2} \| F\|_{H^1(\Gamma)}
     \end{equation}
     and
     \begin{equation}\label{i:appendix-esti-3}
          \|(\mu I + L_{\mathcal T} )^{-1} F\|_{L^2(\Gamma)}\lesssim \mu^{-1} \| F\|_{L^2(\Gamma)}.
     \end{equation}
\end{lemma}
We also note that the arguments leading to \eqref{i:bochner-decay} but using expansions based on the discrete eigenpairs imply that if $r\in \left( (n-1)/2,2s \right)$ and $\mu>0$ we have
\begin{equation}\label{i:appendix-esti-4}
     \|(\mu I + \widetilde L_{\mathcal T})^{-1}\widetilde F\|_{L^2(\gamma)} \lesssim \|\widetilde L^{-r/2}_{\mathcal T}\widetilde F\|_{L^2(\gamma)}
     \left\{
     \begin{aligned}
           & 1,                   &  & \text{when } \mu \le 1, \\
           & \mu^{\tfrac r2 - 1}, &  & \text{when } \mu > 1.
     \end{aligned}
     \right.
\end{equation}
for any $\widetilde F \in \widetilde{\mathbb V}(\mathcal T)$.

To prove Lemma~\ref{l:frac-geo-error}, it suffices to show that
\begin{equation} \label{e:splitting_Eps_l}
     S_i:=k\sum_{l=-\mathtt M}^{\mathtt N} \mu_l^{1-s}\|\mathcal E_l^i \widetilde W^\Psi \|_{L^2(\Omega;L^2(\Gamma))} \le C(h)h^2, \quad i=1,2.
\end{equation}
which we do now by estimating each term separately.

We start with $S_2$ and let $r \in ( \tfrac{n-1}2 , 2s)$.
Thanks to \eqref{i:appendix-esti-3}, the geometric error \eqref{i:sigma}, and \eqref{i:appendix-esti-4}, we obtain
$$
     S_2   \lesssim h^2\|\widetilde L^{-r/2}_{\mathcal T}\widetilde W^\Psi\|_{L^2(\Omega;L^2(\gamma))}
     \bigg(\sum_{\mu_l \le 1} k\mu_l^{1-s} + \sum_{\mu_l > 1} k\mu_l^{-s+r/2}\bigg),
$$
The estimate of $\|\widetilde L^{-r/2}_{\mathcal T}\widetilde W^\Psi\|_{L^2(\Omega;L^2(\gamma))}$ provided by Lemma~\ref{l:negative-boundedness} in conjunction with $-s+r/2<0$, yield
$$
     S_2 \lesssim h^2 \|\widetilde L^{-r/2}_{\mathcal T}\widetilde W^\Psi\|_{L^2(\Omega;L^2(\gamma))} \lesssim h^2,
$$
which is the desired estimate in disguised (with $C(h) \lesssim 1$).

For $S_1$, we first estimate the discrepancy between $\widetilde U_{\mathcal T} := \widetilde T_{\mathcal T}\widetilde F$ and $U_{\mathcal T}:=T_{\mathcal T}\sigma P\widetilde F$ for any $\widetilde F\in \widetilde{\mathbb V}(\mathcal T)$. By definition, see \eqref{e:discrete_T}, $U_{\mathcal T}$ satisfies
\[
     A_{\Gamma}(U_{\mathcal T},V) = \int_{\Gamma}\sigma P\widetilde F V
\]
for any $V\in \mathbb V(\mathcal T)$.
In turn, the change of variable formula \eqref{e:area} together with the definition of $\widetilde U_{\mathcal T}$
imply
$$
     A_{\Gamma}(U_{\mathcal T},V) =  \int_{\gamma} \widetilde F(P^{-1} V)\ = a_{\gamma}(\widetilde U_{\mathcal T},P^{-1} V)
$$
upon realizing that $P^{-1}V \in \widetilde{\mathbb V}(\mathcal T)$.
Consequently, we find that
\[
     \begin{aligned}
          \kappa^2 \int_\Gamma & (P\widetilde U_{\mathcal T} - U_{\mathcal T})V
          + \int_\Gamma \nabla_\Gamma(P\widetilde U_{\mathcal T} - U_{\mathcal T})\cdot\nabla_\Gamma V                    \\
                               & = A_\Gamma(P\widetilde U_{\mathcal T}, V) - a_\gamma(\widetilde U_{\mathcal T}, P^{-1}V) \\
                               & = \kappa^2\int_\gamma (P^{-1}\sigma^{-1}-1) \widetilde U_{\mathcal T}(P^{-1}V)
          + \int_\gamma \nabla_\gamma\widetilde U_{\mathcal T}\cdot \widetilde{\mathbf E}\nabla_\gamma (P^{-1}V),
     \end{aligned}
\]
where the error matrix $\widetilde{\mathbf E}$ satisfies $\|\widetilde{\mathbf E}\|_{L^\infty(\gamma)}\lesssim h^2$ (see e.g. \cite[Corollary~33]{bonito2020finite}). We now set $V=P\widetilde U_{\mathcal T} - U_{\mathcal T}$ so that with Cauchy-Schwarz inequality and the geometric consistency  \eqref{i:sigma}, we get
\begin{eqnarray}\label{i:geo-solution-error}
     \|(P\widetilde T_{\mathcal T}-T_{\mathcal T}\sigma P)\widetilde F\|_{H^1(\Gamma)}=\|P\widetilde U_{\mathcal T} - U_{\mathcal T}\|_{H^1(\Gamma)} & \lesssim & h^2\|\widetilde U_{\mathcal T}\|_{H^1(\gamma)} \nonumber \\
     & \lesssim & h^2 \|\widetilde F\|_{H^{-1}(\gamma)} .
\end{eqnarray}
We next estimate $\|\mathcal E_l^1 \widetilde W^\Psi\|_{L^2(\Omega;L^2(\Gamma))}$ distinguishing two cases.


\boxed{1} If $\mu_l > 1$, we apply successively
\eqref{i:appendix-esti-2}, \eqref{i:geo-solution-error} and \eqref{i:appendix-esti-1} (with $p=1$ and $q=-\min\{1,r\}$) to get
\[
     \begin{aligned}
          \|\mathcal E_l^1 \widetilde W^\Psi\|_{L^2(\Omega;L^2(\Gamma))}
           & \lesssim \mu_l^{\min\left\{-\tfrac12, \tfrac r2 - 1\right\}} h^2 \|\widetilde L_{\mathcal T}^{q/2}\widetilde W^\Psi\|_{L^2(\Omega;L^2(\gamma))} \\
           & \lesssim \mu_l^{\tfrac r2 - 1} h^2 \|\widetilde L_{\mathcal T}^{q/2}\widetilde W^\Psi\|_{L^2(\Omega;L^2(\gamma))},
     \end{aligned}
\]
where we used that $\mu_l^{-1/2}\le \mu_l^{r/2-1}$ if $r>1$.

\boxed{2} If $0 < \mu_l \le 1$, we set again $q=-\min\{1,r\}$. Using to the relation $L_{\mathcal T}(\mu_l I+L_{\mathcal T})^{-1}=I-\mu_l(\mu_l I + L_{\mathcal T})^{-1}$ and \eqref{i:appendix-esti-3} we have
\[
     \begin{aligned}
          \|L_{\mathcal T}(\mu_l I & + L_{\mathcal T})^{-1} F\|_{L^2(\Gamma)}                                                          \\
                                   & \le \|F\|_{L^2(\Gamma)} + \mu_l\|(\mu_l I+L)^{-1}F\|_{L^2(\Gamma)} \lesssim \|F\|_{L^2(\Gamma)} .
     \end{aligned}
\]
Moreover, \eqref{i:geo-solution-error} implies that
\[
     \|(P\widetilde T_{\mathcal T}-T_{\mathcal T}\sigma P)\widetilde F\|_{H^1(\Gamma)}                                              \lesssim h^2\|\widetilde F\|_{H^{q}(\gamma)},
\]
while \eqref{i:appendix-esti-1} with $p=-q$ reads
\[
     \|\widetilde L_{\mathcal T}(\mu_l I + \widetilde L_{\mathcal T})^{-1} \widetilde F\|_{H^{q}(\gamma)} \lesssim \|\widetilde L_{\mathcal T}^{q/2}\widetilde F\|_{L^2(\gamma)} .
\]
Gathering the above three estimates yields
\[
     \|\mathcal E_l^1 \widetilde W^\Psi\|_{L^2(\Omega;L^2(\Gamma))}
     \lesssim h^2 \|\widetilde L_{\mathcal T}^{q/2}\widetilde W^\Psi\|_{L^2(\Omega;L^2(\gamma))} .
\]
and as a consequence
\[
     \begin{aligned}
          S_1 & \lesssim h^2\|\widetilde L_{\mathcal T}^{q/2}\widetilde W^\Psi\|_{L^2(\Omega;L^2(\gamma))}
          \bigg(\sum_{\mu_l \le 1} k\mu_l^{1-s} + \sum_{\mu_l > 1} k\mu_l^{-s+r/2}\bigg)
          \\
              & \lesssim h^2 \|\widetilde L_{\mathcal T}^{q/2}\widetilde W^\Psi\|_{L^2(\Omega;L^2(\gamma))}.
     \end{aligned}
\]
Recall that $q= -\min(1,r)$ and so we now argue depending on whether $r\le 1$ (which can only happen when $n=2$) and $r>1$. When $r\le 1$, we have $q=-r$ and thus $S_1\lesssim h^2$ by Lemma~\ref{l:negative-boundedness}. For $r>1$, i.e. $q=-1$, we write $r=1+\epsilon$ with $\epsilon\in(0,2s-1)$ and we invoke an inverse inequality to write
\[
     S_1 \lesssim h^{2-\epsilon} \|\widetilde L_{\mathcal T}^{-(1+\epsilon)/2}\widetilde W^\Psi\|_{L^2(\Omega;L^2(\gamma))}\lesssim \frac{1+\epsilon}{1+\epsilon-(n-1)/2}h^{2-\epsilon},
\]
where we applied Lemma~\ref{l:negative-boundedness} for the second inequality. To conclude, we choose $\epsilon = (2s-1)/\ln(h^{-1})$ yielding $S_1\lesssim h^2$ for $n=2$ and $S_1\lesssim \ln(h^{-1})h^2$ for $n=3$. The proof is now complete.

\bibliographystyle{plain}
\bibliography{main}

\begin{thebibliography}{10}

\bibitem{AT2007}
Robert~J. Adler and Jonathan~E. Taylor.
\newblock {\em Random Fields and Geometry}.
\newblock Springer Monographs in Mathematics. Springer, 2007.

\bibitem{antil2018fractional}
Harbir Antil, Johannes Pfefferer, and Sergejs Rogovs.
\newblock Fractional operators with inhomogeneous boundary conditions: analysis, control, and discretization.
\newblock {\em Commun. Math. Sci.}, 16(5):1395--1426, 2018.

\bibitem{dealII94}
Daniel Arndt, Wolfgang Bangerth, Marco Feder, Marc Fehling, Rene Gassm{\"o}ller, Timo Heister, Luca Heltai, Martin Kronbichler, Matthias Maier, Peter Munch, Jean-Paul Pelteret, Simon Sticko, Bruno Turcksin, and David Wells.
\newblock The \texttt{deal.II} library, version 9.4.
\newblock {\em J. Numer. Math.}, 30(3):231--246, 2022.

\bibitem{ayachit2015}
Utkarsh Ayachit.
\newblock {\em The ParaView Guide: A Parallel Visualization Application}.
\newblock Kitware, Inc., USA, 2015.

\bibitem{BNT2007}
Ivo Babu\v{s}ka, Fabio Nobile, and Raul Tempone.
\newblock A stochastic collocation method for elliptic partial differential equations with random input data.
\newblock {\em SIAM J. Numer. Anal.}, 45:1005--1034, 2007.

\bibitem{babuska1991eigenvalue}
Ivo Babu\v{s}ka and John~E. Osborn.
\newblock Eigenvalue problems.
\newblock In {\em Handbook of Numerical Analysis, Vol. II}, pages 641--787. Elsevier Science Publishers B.V., North-Holland, 1991.

\bibitem{BTZ2004}
Ivo Babu\v{s}ka, Raul Tempone, and Georgios Zouraris.
\newblock Galerkin finite element approximations of stochastic ellipticpartial differential equations.
\newblock {\em SIAM J. Numer. Anal.}, 42:800--825, 2004.

\bibitem{BD2022}
Markus Bachmayr and Ana Djurdjevac.
\newblock Multilevel representations of isotropic {G}aussian random fields on the sphere.
\newblock {\em IMA J. Numer. Anal.}, 43:1970--2000, 2023.

\bibitem{Balakrishnan60}
Alampallam~V. Balakrishnan.
\newblock Fractional powers of closed operators and the semigroups generated by them.
\newblock {\em Pacific J. Math.}, 10:419--437, 1960.

\bibitem{boffi2010finite}
Daniele Boffi.
\newblock Finite element approximation of eigenvalue problems.
\newblock {\em Acta Numer.}, 19:1--120, 2010.

\bibitem{BK2020}
David Bolin and Kristin Kirchner.
\newblock The rational {SPDE} approach for gaussian random fields with general smoothness.
\newblock {\em J. Comput. Graph. Stat.}, 29(2):274--285, 2020.

\bibitem{BKK2018}
David Bolin, Kristin Kirchner, and Mih{\'a}ly Kov{\'a}cs.
\newblock Weak convergence of {G}alerkin approximations for fractional elliptic stochastic {PDE}s with spatial white noise.
\newblock {\em BIT Numer. Math.}, 58(4):881--906, 2018.

\bibitem{bolin2020numerical}
David Bolin, Kristin Kirchner, and Mih{\'a}ly Kov{\'a}cs.
\newblock Numerical solution of fractional elliptic stochastic {PDE}s with spatial white noise.
\newblock {\em IMA J. Numer. Anal.}, 40(2):1051--1073, 2020.

\bibitem{BSW2022}
David Bolin, Alexandre~B. Simas, and Jonas Wallin.
\newblock {G}aussian {W}hittle-{M}at{\'e}rn fields on metric graphs.
\newblock {\em arXiv preprint arXiv:2205.06163}, 2022.

\bibitem{bonito2016high}
Andrea Bonito, J.~Manuel Casc{\'o}n, Khamron Mekchay, Pedro Morin, and Ricardo~H. Nochetto.
\newblock High-order {AFEM} for the {L}aplace--{B}eltrami operator: Convergence rates.
\newblock {\em Found. Comut. Math.}, 16(6):1473--1539, 2016.

\bibitem{bonito2013afem}
Andrea Bonito, J.~Manuel Casc{\'o}n, Pedro Morin, and Ricardo~H. Nochetto.
\newblock {AFEM} for geometric {PDE}: the {L}aplace--{B}eltrami operator.
\newblock In {\em Analysis and numerics of partial differential equations}, pages 257--306. Springer, 2013.

\bibitem{bonito2020finite}
Andrea Bonito, Alan Demlow, and Ricardo~H. Nochetto.
\newblock Finite element methods for the {L}aplace--{B}eltrami operator.
\newblock In {\em Handbook of Numerical Analysis}, volume~21, pages 1--103. Elsevier, 2020.

\bibitem{bonito2021approximation}
Andrea Bonito and Wenyu Lei.
\newblock Approximation of the spectral fractional powers of the {L}aplace-{B}eltrami operator.
\newblock {\em Numer. Math. Theory Methods Appl.}, 15(4):1193--1218, 2022.

\bibitem{BLP17}
Andrea Bonito, Wenyu Lei, and Joseph~E. Pasciak.
\newblock On sinc quadrature approximations of fractional powers of regularly accretive operators.
\newblock {\em J. Numer. Math.}, 27(2):57--68, 2019.

\bibitem{bonito2015numerical}
Andrea Bonito and Joseph Pasciak.
\newblock Numerical approximation of fractional powers of elliptic operators.
\newblock {\em Math. Comput.}, 84(295):2083--2110, 2015.

\bibitem{bonito2012convergence}
Andrea Bonito and Joseph~E. Pasciak.
\newblock Convergence analysis of variational and non-variational multigrid algorithms for the {L}aplace--{B}eltrami operator.
\newblock {\em Math. Comput.}, 81(279):1263--1288, 2012.

\bibitem{BP15}
Andrea Bonito and Joseph~E. Pasciak.
\newblock Numerical approximation of fractional powers of regularly accretive operators.
\newblock {\em IMA J. Numer. Anal.}, 37(3):1245--1273, 2017.

\bibitem{BN2014}
Francesca Bonizzoni and Fabio Nobile.
\newblock Perturbation analysis for the {D}arcy problem with log-normal permeability.
\newblock {\em SIAM-ASA J. Uncertain.}, 2(1):223--244, 2014.

\bibitem{BTMD2020}
Viacheslav Borovitskiy, Alexander Terenin, Peter Mostowsky, and Marc~P. Deisenroth.
\newblock Mat\'{e}rn {G}aussian processes on {R}iemannian manifolds.
\newblock In {\em Advances in Neural Information Processing Systems}, volume~33, pages 12426--12437. Curran Associates, Inc., 2020.

\bibitem{CGST2011}
K.~Andrew Cliffe, Mike~B. Giles, Robert Scheichl, and Aretha~L. Teckentrup.
\newblock Multilevel monte carlo methods and applications to elliptic pdes with random coefficients.
\newblock {\em Comput. Visual. Sci.}, 14(1):3--15, 2011.

\bibitem{cox2020regularity}
Sonja~G Cox and Kristin Kirchner.
\newblock Regularity and convergence analysis in sobolev and h{\"o}lder spaces for generalized whittle--mat{\'e}rn fields.
\newblock {\em Numer. Math.}, 146(4):819--873, 2020.

\bibitem{delfour2011shapes}
Michel~C. Delfour and J.-P. Zol{\'e}sio.
\newblock {\em Shapes and geometries: metrics, analysis, differential calculus, and optimization}.
\newblock SIAM, 2011.

\bibitem{demlow09}
Alan Demlow.
\newblock Higher-order finite element methods and pointwise error estimates for elliptic problems on surfaces.
\newblock {\em SIAM J. Numer. Anal.}, 47(2):805--827, 2009.

\bibitem{DR2007}
Peter~J. Diggle and Paulo~J. Ribeiro.
\newblock {\em Model-based geostatistics}.
\newblock Springer, 2007.

\bibitem{dziuk88}
Gerhard Dziuk.
\newblock {\em Finite elements for the {B}eltrami operator on arbitrary surfaces}, pages 142--155.
\newblock Springer Berlin Heidelberg, 1988.

\bibitem{Dziuk13}
Gerhard Dziuk and Charles~M. Elliott.
\newblock Finite element methods for surface {PDE}s.
\newblock {\em Acta Numer.}, 22:289--396, 2013.

\bibitem{ern2021finite}
Alexandre Ern, Jean-Luc Guermond, et~al.
\newblock {\em Finite Elements II}.
\newblock Springer, 2021.

\bibitem{FLH2015}
Aasa Feragen, Fran\c{c}ois Lauze, and S\o{}ren Hauberg.
\newblock Geodesic exponential kernels: When curvature and linearity conflict.
\newblock In {\em 2015 IEEE Conference on Computer Vision and Pattern Recognition (CVPR)}, pages 3032--3042, 2015.

\bibitem{FST2005}
Philipp Frauenfelder, Christoph Schwab, and Radu~A. Todor.
\newblock Finite elements for elliptic problems with stochastic coefficients.
\newblock {\em Comput. Methods Appl. Mech. Eng.}, 194(2--5):205--228, 2005.

\bibitem{GS1991}
Roger~G. Ghanem and Pol~D. Spanos.
\newblock {\em Stochastic Finite Elements: A Spectral Approach}.
\newblock Springer, New-York, 1991.

\bibitem{GKNSS2011}
Ivan~G. Graham, Frances~Y. Kuo, Dirk Nuyens, Robert Scheichl, and Ian~H. Sloan.
\newblock Quasi-monte carlo methods for ellipticpdes with random coefficients and applications.
\newblock {\em J. Comput. Phys.}, 230(10):3668--3694, 2011.

\bibitem{HHKS2021}
Helmut Harbrecht, Lukas Herrmann, Kristin Kirchner, and Christoph Schwab.
\newblock Multilevel approximation of {G}aussian random fields: Covariance compression, estimation and spatial prediction.
\newblock {\em arXiv preprint arXiv:2103.04424}, 2021.

\bibitem{herrmann2020multilevel}
Lukas Herrmann, Kristin Kirchner, and Christoph Schwab.
\newblock Multilevel approximation of gaussian random fields: fast simulation.
\newblock {\em Math. Models Methods Appl. Sci.}, 30(01):181--223, 2020.

\bibitem{ivrii2016100}
Victor Ivrii.
\newblock 100 years of weyl’s law.
\newblock {\em Bull. Math. Sci.}, 6(3):379--452, 2016.

\bibitem{jansson2021surface}
Erik Jansson, Mih{\'a}ly Kov{\'a}cs, and Annika Lang.
\newblock Surface finite element approximation of spherical {W}hittle--{M}at{\'e}rn {G}aussian random fields.
\newblock {\em SIAM J. Sci. Comput.}, 44(2):A825--A842, 2022.

\bibitem{Kato61}
Tosio Kato.
\newblock Fractional powers of dissipative operators.
\newblock {\em J. Math. Soc. Japan}, 13:246--274, 1961.

\bibitem{knyazev2006new}
Andrew~V. Knyazev and John~E. Osborn.
\newblock New a priori {FEM} error estimates for eigenvalues.
\newblock {\em SIAM J. Numer. Anal.}, 43(6):2647--2667, 2006.

\bibitem{LP2021}
Annika Lang and Mike Pereira.
\newblock Galerkin--{C}hebyshev approximation of {G}aussian random fields on compact {R}iemannian manifolds.
\newblock {\em BIT Numer. Math.}, 63(51), 2023.

\bibitem{LS2015}
Annika Lang and Christoph Schwab.
\newblock Isotropic gaussian random fields on the sphere: regularity, fast simulation and stochastic partial differential equations.
\newblock {\em Ann. Appl. Probab.}, 25(6):3047--3094, 2015.

\bibitem{KL2010}
Olivier Le~Ma{\^i}tre and Omar~M. Knio.
\newblock {\em Spectral Methods for Uncertainty Quantification: With Applications to Computational Fluid Dynamics}.
\newblock Springer, 2010.

\bibitem{LBR2022}
Finn Lindgren, David Bolin, and H\r{a}vard Rue.
\newblock The {SPDE} approach for {G}aussian and non-{G}aussian fields: 10 years and still running.
\newblock {\em Spat. Stat.}, 50:100599, 2022.

\bibitem{lindgren2011explicit}
Finn Lindgren, H{\aa}vard Rue, and Johan Lindstr{\"o}m.
\newblock An explicit link between {G}aussian fields and {G}aussian {M}arkov random fields: the stochastic partial differential equation approach.
\newblock {\em J. R. Stat. Soc. Ser. B Methodol.}, 73(4):423--498, 2011.

\bibitem{L1977}
Michel Lo{\`e}ve.
\newblock {\em Probability theorey I (4th ed.)}.
\newblock Volume 45 of Graduate Texts in Mathematics. Springer, New York, 1977.

\bibitem{L1978}
Michel Lo{\`e}ve.
\newblock {\em Probability theorey II (4th ed.)}.
\newblock Volume 46 of Graduate Texts in Mathematics. Springer, New York, 1978.

\bibitem{LPS2014}
Gabriel~J. Lord, Catherine~E. Powell, and Tony Shardlow.
\newblock {\em An Introduction to Computational Stochastic {PDE}s}.
\newblock Cambridge Texts in Applied Mathematics. Cambridge University Press, 2014.

\bibitem{lund1992sinc}
John Lund and Kenneth~L. Bowers.
\newblock {\em Sinc methods for quadrature and differential equations}.
\newblock SIAM, 1992.

\bibitem{MP2011}
Domenico Marinucci and Giovanni Peccati.
\newblock {\em An Introduction to Computational Stochastic {PDE}s}.
\newblock Cambridge University Press, 2011.

\bibitem{M1960}
Bertil Mat{\'e}rn.
\newblock {\em Spatial variation : Stochastic models and their application to some problems in forest surveys and other sampling investigations}.
\newblock Meddelanden Fr\r{a}n Statens Skogsforskningsinstitut, Band 49, Nr. 5, Stockholm, 1960.

\bibitem{mekchay2011afem}
Khamron Mekchay, Pedro Morin, and Ricardo~H. Nochetto.
\newblock {AFEM} for the laplace-beltrami operator on graphs: design and conditional contraction property.
\newblock {\em Math. Comput.}, 80(274):625--648, 2011.

\bibitem{M1909}
James Mercer.
\newblock Functions of positive and negative type and their connection with the theory of integral equations.
\newblock {\em Philos. Trans. R. Soc. A}, 209(441--458):415--446, 1909.

\bibitem{NT2015}
Fabio Nobile and Francesco Tesei.
\newblock A multi level {M}onte {C}arlo method with control variate for elliptic {PDE}s with log-normal coefficients.
\newblock {\em Stoch. Partial Differ. Equ.: Anal.}, 3:398--444, 2015.

\bibitem{RW2006}
Carl~E. Rasmussen and Christopher K.~I. Williams.
\newblock {\em Gaussian Processes for Machine Learning}.
\newblock MIT Press, Cambridge, 2006.

\bibitem{ST2006}
Christoph Schwab and Radu~A. Todor.
\newblock {K}arhunen–{L}o{\`e}eve approximation of random fields by generalized fast multipole methods.
\newblock {\em J. Comput. Phys.}, 7217(1):100--122, 2006.

\bibitem{S1999}
Michael~L. Stein.
\newblock {\em Interpolation of Spatial Data: Some Theory for Kriging}.
\newblock Springer Series in Statistics. Springer, New York, 1999.

\bibitem{W1954}
Peter Whittle.
\newblock On stationary processes in the plane.
\newblock {\em Biometrika}, 41:434--449, 1954.

\bibitem{W1963}
Peter Whittle.
\newblock Stochastic processes in several dimensions.
\newblock {\em Bull. Int. Stat. Inst.}, 40:974--994, 1963.

\bibitem{XH2005}
Dongbin Xiu and Jan~S. Hesthaven.
\newblock High-order collocation methods for differential equations with random inputs.
\newblock {\em SIAM J. Sci. Comput.}, 27:1118--1139, 2005.

\end{thebibliography}

\end{document}